\documentclass[11pt,a4,leqno]{amsart}
\usepackage{amsthm}
\usepackage{amsmath}
\usepackage{amssymb}
\usepackage{amsfonts}
\usepackage{mathrsfs}
\usepackage{graphicx}
\usepackage{bbm}
\usepackage[colorlinks=true, citecolor=blue, linkcolor=blue, urlcolor=red]{hyperref}
\usepackage{color}
\usepackage[all]{xy}
\usepackage{comment}
\usepackage{here}
\usepackage{amscd}
\usepackage{caption}
\usepackage{listings}
\usepackage{setspace}
\usepackage{booktabs}
\usepackage{parskip}
\usepackage{bbm}
\usepackage[dvipsnames]{xcolor}
\allowdisplaybreaks
\usepackage[margin=1in]{geometry}

\newcommand{\mathsym}[1]{{}}
\newcommand{\unicode}[1]{{}}

\makeatletter

\@addtoreset{equation}{section}
\makeatother

\makeatletter
\@namedef{subjclassname@2020}{%
  \textup{2020} Mathematics Subject Classification}
\makeatother

\newtheoremstyle{my theoremstyle}
{1.0em}                    
    {1.0em}                    
    {\itshape}                   
    {}                           
    {\scshape}                   
    {.}                          
    {.5em}                       
    {}  

\newtheoremstyle{dfn}
{1.0em}                    
    {1.0em}                    
    {}                   
    {}                           
    {\scshape}                   
    {.}                          
    {.5em}                       
    {}  
    
\theoremstyle{my theoremstyle}
   \newtheorem{thm}{Theorem}[section]
   \newtheorem{lem}[thm]{Lemma}
   \newtheorem{prop}[thm]{Proposition}
   \newtheorem{cor}[thm]{Corollary}
   \newtheorem{conj}[thm]{Conjecture}
\theoremstyle{dfn}
   \newtheorem{dfn}[thm]{Definition}
   
\theoremstyle{remark}   
   
   \newtheorem{rmk}[thm]{{\scshape Remark}}

\newcommand{\Z}{\mathbb{Z}}
\newcommand{\Q}{\mathbb{Q}}
\newcommand{\R}{\mathbb{R}}
\newcommand{\C}{\mathbb{C}}

\newcommand{\e}{\varepsilon}
\newcommand{\z}{\zeta}
\newcommand{\al}{\alpha}

\numberwithin{equation}{section}

\newcommand{\Spec}{\operatorname{Spec}}

\newcommand{\ord}{\operatorname{ord}}

\newcommand{\dlog}{\operatorname{dlog}}
\newcommand{\darg}{\operatorname{darg}}

\date{\today}

\begin{document} 
\title{Elements in $K_4$ and regulator maps of Fermat curves}
\author{Fran\c{c}ois Brunault}
\author[David Lilienfeldt]{David T.-B. G. Lilienfeldt}
\author{Yusuke Nemoto}
\date{\today}
\address{FB: Unité de mathématiques pures et appliquées (UMPA), ENS de Lyon, 46 allée d'Italie, 69007 Lyon, France.}
\email{francois.brunault@ens-lyon.fr}
\address{DL: Centre de Mathématiques Laurent Schwartz (CMLS), CNRS, École polytechnique, Institut Polytechnique de Paris, Palaiseau, France.} 
\email{david.lilienfeldt@polytechnique.edu}
\address{YN: Graduate School of Science, Chiba University, 
Yayoicho 1-33, Inage, Chiba, 263-8522, Japan.}
\email{y-nemoto@waseda.jp}

\begin{abstract}
    We construct explicit elements in the group $K_4^{(3)}$ of the Fermat curves $x^N+y^N=1$ for all $N\geq 3$. The construction, which is uniform in $N$, uses polylogarithmic complexes and a map of de Jeu to $K$-theory. We prove that the elements are non-trivial by showing that their images under Beilinson's regulator map are non-zero. Notably, we obtain explicit formulas for their regulator integrals involving special values of Zagier's trilogarithm function. As a corollary, we show that these regulator integrals are asymptotic to $\frac{3}{2}\zeta(3)N^2$ as $N\to +\infty$. Moreover, we derive formulas for the regulators of our elements in terms of hypergeometric functions, generalizing results of Otsubo for $K_2$ groups of Fermat curves. Finally, we numerically verify some cases of Beilinson's conjectures on special values of $L$-functions at $s=3$ for $N\in \{ 3,4,6 \}$. 
\end{abstract}

\maketitle

\tableofcontents

\section{Introduction}

For any positive integer $N$, we let $X_N$ denote the (plane, smooth, projective) Fermat curve of degree $N$ over $\Q$ with the affine equation
$x^N+ y^N=1.$ In this paper, we are interested in the third Adams eigenspace $K_4^{(3)}(X_N)$ of $K_4(X_N)\otimes \Q$ 
and its relation to $L$-values in connection with Beilinson's conjecture \cite{Bei80}. 
This work can be viewed as a generalization of works of Ross \cite{Ross94} and Otsubo \cite{Ots11} on $K_2$ of Fermat curves, and it is the first to tackle Beilinson's conjecture for general Fermat curves beyond $K_2$. 

Recall that the group $K_2(\Q(X_N))$ can be described, using Matsumoto's theorem, as
\[
K_2(\Q(X_N)) \simeq \frac{\Q(X_N)^\times \otimes_{\Z} \Q(X_N)^\times}{\langle f \otimes (1-f) \colon f\in \Q(X_N)\setminus \{ 0,1 \} \rangle}.
\]
The group $K_2(X_N) \otimes \Q$ can then be identified with the subgroup of $K_2(\Q(X_N)) \otimes \Q$ obtained by taking the kernel of the tame symbols (see \cite[Section 1]{Ross94} for details). Moreover, we have $K_2(X_N) \otimes \Q = K_2^{(2)}(X_N)$ (see \cite[Section 7.4]{Nek94}) .

The description of the $K_4$ group of $X_N$ is more involved. However, it is expected that $K_4(X_N)\otimes \Q = K_4^{(3)}(X_N)$
\footnote{In motivic cohomology notations, we have $K_4^{(n)}(X_N) = H^{2n-4}_{\mathcal{M}}(X_N, \Q(n))$. According to Beilinson, the spectral sequence for \textrm{Ext} groups in the conjectural category of mixed motives over $\Q$ degenerates to motivic cohomology \cite[Section 3.6]{Nek94}. In this category, Beilinson conjectures that $\mathrm{Ext}^i$ vanishes for $i>1$. This gives rise to short exact sequences as explained in \emph{loc.~cit.}, and would imply that $H^{2n-4}_{\mathcal{M}}(X_N, \Q(n)) = 0$ for $n \neq 3$.}
and a conjectural description of $K_4^{(3)}$ of fields close in spirit to Matsumoto's theorem has been proposed by Goncharov \cite{Goncharov} and de Jeu \cite{Jeu95,Jeu96} using polylogarithmic complexes. De Jeu constructed a map \cite{Jeu96} from the cohomology in degree $2$ of his weight $3$ polylogarithmic complex to $K_4^{(3)}$ of the given field. This map enables us to construct explicit elements  \footnote{De Jeu's map is defined up to a universal sign, so that our $K_4$ elements are only defined up to this sign. This is the reason for the sign ambiguity in our regulator formulas.}
\[
\Xi_{N}:=\Xi_{N}^{-2,1}\in K_4^{(3)}(X_N)
\] 
(see Corollary \ref{cor:K4open} and Theorem \ref{prop:res1}). 
The superscript $(-2,1)$ relates to a specific submotive of $h^1(X_N)$ to which the element belongs (see Corollary \ref{cor:pNrs}). The main theorem is the following. 

\begin{thm}\label{thm1}
    For all $N\geq 3$, the element $\Xi_N$ is non-trivial in $K_4^{(3)}(X_N)$.
\end{thm}

This result is reminescent of non-triviality results for the Ross element $\{ 1-x, 1-y\}\in K_2^{(2)}(X_N)$ \cite[Theorem 1]{Ross94}. 
Our construction of the elements $\Xi_N$, which is detailed in Section \ref{s:2coc}, is inspired by a construction of de Jeu \cite{JeuSlides} for a CM elliptic curve of conductor $27$ that is $3$-isogenous to $X_3$ (see Remark \ref{rmk:jeuX3}).

Theorem \ref{thm1} provides non-trivial $K_4$ elements for an infinite family of curves indexed by the degree~$N$, and the construction is uniform in $N$. This is similar to the construction by the first-named author of $K_4$ elements for the family of modular curves $X(N)$ and $X_1(N)$ \cite{Bru20}. 

One can hope to use automorphisms of $X_{N, \overline{\Q}}:=X_N  \times_{\Q} \overline{\Q}$ to produce more $K_4$ elements. An obvious symmetry defined over $\Q$ is given by the map
$
    \iota \colon (x,y) \mapsto (y, x).
$
Let $\e$ denote a $2N$-th root of unity satisfying $\e^N=-1$. Another automorphism, defined this time over $\Q(\e)$, is given by 
$
\tau_\e \colon (x,y) \mapsto \left(\frac{1}{x}, \frac{\varepsilon y}{x} \right).
$
The following result is Corollary \ref{cor:indep} in the main text. 

\begin{thm}\label{thm2}
    If $N\geq 3$ is not divisible by $3$, then the elements $\Xi_N, \iota^*(\Xi_N)$, $\tau_{\e}^*(\Xi_N)$ are non-zero and linearly independent in $K_4^{(3)}(X_{N,\Q(\e)})$, where $X_{N,\Q(\e)}:=X_N  \times_{\Q} \Q(\e)$.
\end{thm}

\subsection{Regulator maps}

We prove Theorems \ref{thm1} and \ref{thm2} by studying the images of our elements under Beilinson's regulator map \cite{Bei80,Nek94}
\begin{equation*}
    r_B : K_4^{(3)}(X_N) \longrightarrow H^1(X_N(\C), \R(2))^+,
\end{equation*}
where $+$ denotes the invariants with respect to the complex conjugation acting on the complex points $X_N(\C)$ and on the coefficients. 

As before, we let $\e$ denote a $2N$-th root of unity with $\e^N=-1$. 
In order to prove Theorem \ref{thm1}, we show that $r_B ((\tau_\e)^*(\Xi_N))$ is non-zero. To do this, we integrate this cohomology class over a closed path $\gamma_N$ in $X_N$ defined by \eqref{def:gammaN} whose homology class generates $H_1(X_N(\C),\Q)$ as a cyclic $\Q[(\Z/N\Z)^{\oplus 2}]$-module. The precise formula is the following. 

\begin{thm}\label{thm3}
     For all $N\geq 3$, we have the equality (up to sign)
    \begin{equation*}
    \int_{\gamma_N} r_B ((\tau_\e)^*(\Xi_N))=\frac43  \int_0^1 \left(f_{N, \e}(t) - \frac{f_{1,-1}(t)}{N}\right) dt
      + \left(\frac1N \mathcal{L}_3(-1/4)-N^2 \mathcal{L}_3(\e (1/4)^{1/N})\right),
\end{equation*}
where $\mathcal{L}_3$ is Zagier's trilogarithm \eqref{GonchL3} and, for $x\in [0,1[$ and $z\in S^1$, the function $f_{N, z}$ is given by
\begin{equation*}
    f_{N,z}(x):=\Re(\log(1- z (x(1-x))^{1/N})) \frac{\log(1-x)}{x}.
\end{equation*}
\end{thm}

The appearance of Zagier's trilogarithm in the above formula is quite remarkable. Its special values are connected with $\zeta(3)$ and, as a corollary of Theorem \ref{thm4}, we obtain asymptotic formulas for regulator integrals.

\begin{thm}\label{thm4}
    Let $\e=-1$ if $N$ is odd and $\e=-e^{\frac{\pi i}{N}}$ if $N$ is even. We have the asymptotic formula (up to sign)
    \[
    \int_{\gamma_N} r_B ((\tau_\e)^*(\Xi_N)) \sim \frac{3}{2} \zeta(3) N^2, \qquad \text{ as } N\to +\infty.
    \]
\end{thm}

An immediate corollary of Theorem \ref{thm4} is that $\Xi_N$ is non-trivial for large $N$. In order to prove Theorem \ref{thm1}, a more careful analysis of the regulator integral formula in Theorem \ref{thm3} is required. This analysis is carried out in the proof of Theorem \ref{thm:nontrivial} and leads to the following statement (which is stronger than Theorem \ref{thm1}). 

\begin{thm}\label{thm5}
    For all $N\geq 3$, $r_B ((\tau_\e)^*(\Xi_N))$ is non-zero.
\end{thm}

We further obtain a closed formula for a differential form representing $r_B ((\tau_\e)^*(\Xi_N))$, which is used in the proof of Theorem \ref{thm2}.
The following result can be viewed as a generalization, from the setting of $K_2$ groups to $K_4$ groups of Fermat curves, of the main result of Otsubo \cite[Theorem 4.14]{Ots11}.

\begin{thm}\label{thm6}
For all $N\geq 3$, we have the equality (up to sign) 
\[
r_B ((\tau_\e)^*(\Xi_N))=\dfrac1{3} \sum_{1 \leq a <N} \e^{-a} \frac{\Gamma(2a)}{\Gamma(a)^2} \left(F\left(\dfrac{a}{N};-1 \right) - F\left(1-\dfrac{a}{N};-1 \right)\right)  x^a y^{a-N} \frac{dx}{x},
\]
where, for any positive real number $\al$, $F(\al;  x)$ is the hypergeometric function given by 
\begin{multline*}
F(\al;  x)=  \dfrac{\Gamma\left(\al \right)^2}{\Gamma\left(2 \al \right)}\left( \sum_{i, n \ge 0} \left(\dfrac{2}{n+1} - \dfrac{1}{\al+i} \right) \dfrac{(\al, i)(\al, n+i+1)}{(2\al, n+2i+2)} x^i \right. \\ 
\left. -\sum_{i, m, n \ge 0}  \left(\dfrac1{m+1}+ \dfrac1{n+1} \right) \dfrac{(\al, i)(\al, m+n+i+2) }{(2\al, m+n+2i+3)} x^i \right),
\end{multline*}
and $(\al, i) := \Gamma(\al+i)/\Gamma(\al)$ denotes the rising Pochhammer symbol. 
\end{thm}

\subsection{Beilinson's conjecture}

Beilinson's conjecture \cite{Bei80} relates the $K$-theory of smooth projective varieties over number fields to special values of their $L$-functions (see Section \ref{subsec:beilinson} for details in the case of curves). In our case, it predicts that the map $r_B\otimes \R$ is an isomorphism with determinant (in suitable bases) equal to $L(X_N, 3)$ up to a non-zero rational scalar and a specific power of $\pi$ (see Section \ref{subsec:beilinson}).
Here, the $L$-function $L(X_N,s):=L(h^1(X_N), s)$ is the one associated with the compatible family of $\ell$-adic $\mathrm{Gal}(\overline{\Q}/\Q)$-representations $\{ H^1_{\mathrm{et}}(X_{N,\overline{\Q}}, \Q_\ell) \}_\ell$. It is given by a convergent Euler product in the region $\Re(s)> 3/2$ and satisfies a functional equation centered at $s=1$ (due to the fact that it is a product of Hecke $L$-functions, see \eqref{eq:HeckeL}). Equivalently, by the functional equation, the determinant is predicted to equal $L^{(g)}(X_N, -1)$ up to a non-zero rational scalar and a specific power of $\pi$, where $g$ denotes the genus of $X_N$.
A version of Beilinson's conjecture for Chow motives exists, and we now make it explicit for Fermat curves.

Following Otsubo \cite{Ots11}, in Section \ref{subsec:motives} we consider a motivic decomposition 
\begin{align*}
h^1(X_{N}) = \bigoplus_{[a, b] \in (\Z/N\Z)^\times \backslash I_N} X_N^{[a, b]},  
\end{align*}
where $I_N \subset (\Z/N\Z)^{\oplus 2}$ consists of pairs $(a,b)$ satisfying $a,b,a+b\neq 0$, and the $X_N^{[a, b]}$ are certain rank $\varphi(N)$ motives over $\Q$ with $\Q$-coefficients (see Corollary \ref{cor:rank}) that exist due to $X_N$ having a large automorphism group. The $L$-functions of these motives are Hecke $L$-functions (see \eqref{eq:HeckeL}) and as a result their analytic properties are well understood.
A weak version of Beilinson's conjecture for these motives can be stated as follows. 
\begin{conj}[Beilinson]\label{conj:weakBeiintro}
    Let $N \geq 3$, $(a,b) \in I_N$, and $k := \varphi(N)/2$. Then there exist $\xi_1, \ldots, \xi_{k}$ in $K_4^{(3)}(X_N^{[a,b]})$, a $\Q$-basis $(\gamma_1, \ldots, \gamma_{k})$ of $H_1(X_N^{[a,b]}, \Q)^+$, and a constant $c \in \Q^\times$ such that
    \begin{equation*} 
        L^{(k)}(X_N^{[a,b]},-1) = c \pi^{-2k} \det \Bigl(\int_{\gamma_i} r_B(\xi_j) \Bigr)_{1 \leq i,j \leq k}.
    \end{equation*}
\end{conj}
In Section \ref{sec:num}, we explain how to numerically evaluate the regulator integrals 
$\int_{\gamma_N} r_B(\Xi_N)$.
We then proceed to numerically verify Conjecture \ref{conj:weakBeiintro} using our element $\Xi_N$ in the cases where $X^{[-2,1]}_N$ has rank $2$, \emph{i.e.}~when $N\in \{ 3,4,6 \}$. Our calculations are performed using PARI/GP \cite{PARI} and the notation $\overset{?}{=}$ below means equality up to at least $35$ digits.
\begin{thm}\label{thm7}
    Numerically, we have the equalities (up to sign):
    \[
    \begin{split}
        &\int_{\gamma_3} r_B(\Xi_3) \stackrel{?}{=} -\frac85 \pi^2 L'(X_3^{[-2,1]},-1), \\
        &\int_{\gamma_4} r_B(\Xi_4) \stackrel{?}{=} -\frac85 \pi^2 L'(X_4^{[-2,1]},-1), \\
        &\int_{\gamma_6} r_B(\Xi_6^{\mathrm{new}}) \stackrel{?}{=} - \frac4{15} \pi^2 L'(X_6^{[-2,1]},-1).
    \end{split}
    \]
\end{thm}
The first numerical equality recovers a calculation of de Jeu \cite{JeuSlides}. In the third equality, $\Xi_6$ splits into the sum of an ``old'' part coming from $X_3$ and a ``new'' part $\Xi_6^{\mathrm{new}}:=(p_6^{[-2,1]})^* \xi_6^{-2,1}$, where $p_6^{[-2,1]}$ is the projector that cuts out the motive $X_6^{[-2,1]}$. In all three cases, the motives are isomorphic to $h^1(E_N)$ for some explicit elliptic curve $E_N$ over $\Q$ with CM (see Proposition \ref{elliptic_curve_isom}). In order to compute the $L$-values, we identify these elliptic curves and use existing $L$-function implementations for elliptic curves. Alternatively (and necessarily in higher rank cases) it is possible to compute the $L$-values as special values of Hecke $L$-functions using Hecke $L$-function implementations. 

\subsection{Prior works}

Beilinson \cite{Bei86} constructed elements in $K^{(n)}_{2n-2}(Y_1(N))$, where $Y_1(N)$ is the modular curve of level $\Gamma_1(N)$ and $n \geq 2$ is an arbitrary integer. The construction uses his theory of the Eisenstein symbol, and the images of these elements under the Beilinson regulator map are related to $L$-values of modular forms of weight $2$ at $s=n$. In particular, by Wiles' modularity theorem, the weak form of Beilinson's conjecture for $L(E,n)$ is known to hold for all elliptic curves $E$ defined over $\Q$. The elements in $K^{(n)}_{2n-2}(E)$ are obtained via push-forward along the modular parametrization $X_1(N) \to E$, and are therefore not explicit. Deninger \cite{Deninger90} constructed explicit elements in $K$-groups of elliptic curves with CM by imaginary quadratic fields and used them to verify the weak form of Beilinson's conjecture. In particular, the weak form of Beilinson's conjecture is known for the motives considered in Theorem \ref{thm7}. But for more general Fermat motives, not much is known. Neither Deninger nor Beilinson used polylogarithmic complexes in their constructions. In recent work \cite{Bru20}, the first-named author constructed explicit elements in $K_4^{(3)}(Y_1(N))$ via the Goncharov--de Jeu method, and used these elements to numerically verify the weak version of Beilinson's conjecture on $L(E,3)$ for all elliptic curves $E$ over $\Q$ of conductor $\leq 50$. Ross \cite{Ross94} introduced a non-trivial element of $K_2^{(2)}(X_N)$ for Fermat curves (as mentioned already above). 
Otsubo \cite{Ots11} gave explicit formulas for the image of the Ross element under regulator maps in terms of hypergeometric functions and numerically verified the weak version of Beilinson's conjecture for some Fermat motives. The present work is the first to tackle Beilinson's conjecture for general Fermat curves beyond $K_2$. 

\subsection{Outline}

In Section \ref{s:fermat}, we recall generalities about Fermat curves, motives, and quotients following Otsubo \cite{Ots11}. 
The relevant polylogarithmic complexes of curves are reviewed in Section \ref{sec:cocycles}, as well as de Jeu's map to $K_4^{(3)}$. The key result for our construction is Theorem \ref{rmk:K4X}. In Section \ref{s:2coc}, we construct explicit weight $3$ Goncharov $2$-cocycles for Fermat curves. Not all of these cocycles give rise to elements of $K_4^{(3)}(X_N)$, but those that do are singled out in Section \ref{sec:elements} (see Theorem \ref{prop:res1}). In Section \ref{sec:asy}, we compute the Goncharov regulator integrals of our elements, proving Theorems \ref{thm3}, \ref{thm4}, and \ref{thm5}. Section \ref{s:hyperg} contains the proof of Theorem \ref{thm6} and Section \ref{sec:num} is concerned with the numerical verification of Beilinson's conjecture (Theorem \ref{thm7}).

\subsection*{Acknowledgements} It is a pleasure to thank Rob de Jeu for answering some of our questions and for explaining the proof of Lemma \ref{rel2}.
This article was written while the third-named author was visiting ENS de Lyon. He would like to thank UMPA for their hospitality. 
The third-named author is supported by JSPS KAKENHI Grant (JP26K16957) and Waseda University Grant for Special Research Projects (Project number: 2026C-262). 

\section{Fermat curves, motives, and quotients} \label{s:fermat}
\subsection{Fermat curves}
For $N\geq 1$, let $X_N$ be the Fermat curve of degree $N$ over $\Q$ defined by the homogeneous equation 
$$X_N \colon x_0^N+ y_0^N=z_0^N. $$
The genus of $X_N$ is $(N-1)(N-2)/2$. We define the cusps of $X_N$ to be the points $(x_0:y_0:z_0)$ satisfying $x_0 y_0 z_0 = 0$. 
Let $\varepsilon$ denote a $2N$-th root of unity such that $\varepsilon^N=-1$ and let $K=\Q(\varepsilon)$.
The Fermat curve has natural symmetries defined over $K$ (the first one listed being defined over $\Q$) given by the maps 
\begin{align*}
\iota & \colon (x_0 : y_0 : z_0) \longmapsto (y_0 : x_0 : z_0) \\
\sigma & \colon (x_0 : y_0 : z_0) \longmapsto (x_0 : \varepsilon^{-1} z_0 : \varepsilon^{-1} y_0) \\
\tau & \colon (x_0 : y_0 : z_0) \longmapsto (\varepsilon^{-1} z_0 : y_0 : \varepsilon^{-1} x_0).
\end{align*}
Note that we have suppressed the dependency on $\e$ from the notations.
In terms of the affine coordinates $x=x_0/z_0$ and $y=y_0/z_0$, the above maps are given by
\begin{align} \label{sym:1}
\iota & \colon (x,y) \longmapsto (y, x) \\ \label{sym:2} 
\sigma & \colon (x,y) \longmapsto \left(\frac{\varepsilon x}{y}, \frac{1}{y} \right) \\  \label{sym:3}
\tau & \colon (x,y) \longmapsto \left(\frac{1}{x}, \frac{\varepsilon y}{x} \right). 
\end{align}
If $N'$ divides $N$, with $N = dN'$, then we have a natural map of curves defined over $\Q$
\begin{equation}\label{map:divideN}
    \pi_{N, N'} \colon X_N \longrightarrow X_{N'}; \quad (x_0 : y_0 : z_0) \longmapsto (x_0^d : y_0^d : z_0^d). 
\end{equation}

\subsection{Fermat motives} \label{subsec:motives}
In this section, we recall the definition of Fermat motives. We refer to \cite[Section 2]{Ots11} for details. 

Denote by $\mathrm{Corr}^0(X_N,X_N)$ the ring of correspondences with $\Q$-coefficients. The identity element $1 \in \mathrm{Corr}^0(X_N,X_N)$ for the composition is the image of $X_N$ under the diagonal embedding $X_N\to X_N\times X_N$. The motive of $X_N$ is $h(X_N):=(X_N, 1)$ and it is defined over $\Q$ with $\Q$-coefficients.

Fixing a base point $P:=(0,1)\in X_N(\Q)$, we define idempotent correspondences $e^i \in \mathrm{Corr}^0(X_N,X_N)$ for $i=0,1,2$ as follows:
\[
e^0 = \{ P \} \times X_N, \qquad e^2 =  X_N \times \{ P \}, \qquad e^1=1-e^0-e^2.
\]
The $i$-th cohomological motive of $X_N$ is then $h^i(X_N):=(X_N, e^i)$ and we clearly have a decomposition of motives over $\Q$ with $\Q$-coefficients
\begin{equation} \label{decomp h0h1h2}
h(X_N)\simeq h^0(X_N)\oplus h^1(X_N) \oplus h^2(X_N) \simeq h(\mathbb{P}^1) \oplus h^1(X_N), 
\end{equation}
where the first isomorphism depends on the base point $P$.

Because $X_N$ has many automorphisms defined over $\Q(\zeta_N)$, we may use these to further decompose the cohomological idempotent correspondences.
Given a non-zero element $(a,b)\in G_N := \Z/N\Z \oplus \Z/N\Z$,  we write $\langle a \rangle$ and $\langle b \rangle$ for the respective integers  in $[ 0, N [\cap \Z$ representing $a$ and $b$. 
We denote an element $(a, b) \in G_N$ also by $g_N^{a, b}$, and write the addition multiplicatively:
$$g_N^{a, b} g_N^{a', b'}=g_N^{a+a', b+b'}.$$
We define a subset of $G_N$ by 
\[
I_N:= \{ (a,b) \in G_N \mid a,b,a+b\neq 0 \},
\]  
which is a set of cardinality $(N-1)(N-2)$. 
Fix a primitive $N$-th root of unity and let $G_N$ act on $X_{N, \Q(\zeta_N)}:=X_N \times_{\Spec{\Q}} \Spec{\Q(\zeta_N)}$ by  
\[
g_N^{a,b}(x,y):=(\zeta_N^a x,\zeta_N^b y).
\]
\begin{dfn} \label{def:pNab}
For $(a, b) \in G_N$, define 
a correspondence of $X_{N, \Q(\zeta_N)}$ with $\Q(\zeta_N)$-coefficients by 
$$p_N^{a, b}= \dfrac1{N^2} \sum_{(r, s) \in G_N} \zeta_N^{-ar-bs} \Gamma_{g_N^{r, s}}, $$
where $\Gamma_g$ denotes the transpose of the graph of $g$. 
\end{dfn}

    As in \cite[(2.5)]{Ots11}, it is not difficult to verify that 
\begin{align} \label{projector1}
\sum_{(a, b) \in G_N} p_N^{a, b}=1 \qquad \text{ and } \qquad p_N^{a, b} p_N^{c, d}= 
 \left\{
\begin{array}{ll}
p_N^{a, b},& \text{ if } (a, b)=(c, d) ;\\
0, & \text{ otherwise.} 
\end{array}
\right.
\end{align}
The elements $p_N^{a,b}$ are thus idempotents in the ring of correspondences, and we may therefore define a motive $X_N^{a, b}:=(X_N, p_N^{a, b})$ over $\Q(\z_N)$ with $\Q(\zeta_N)$-coefficients. It is clear from \eqref{projector1} that we have a decomposition of motives 
\[
h(X_{N, \Q(\z_N)}) \otimes \Q(\zeta_N)\simeq \bigoplus_{(a, b) \in G_N} X_N^{a, b}.
\]
We can be more precise. In fact, by \cite[Propositon 2.9]{Ots11}, $p_N^{a,b}=0$ if only one of $a,b, a+b$ is $0$, and 
\begin{equation*}
    p_N^{0,0}=e^0 + e^2 \qquad \text{ and } \qquad \sum_{(a,b)\in I_N} p_N^{a,b}=e^1.
\end{equation*}
We record these equalities in a proposition:

\begin{prop}\label{lem:mot_decomp}
    We have $X_N^{0,0} \simeq h(\mathbb{P}^1)$ and $X_N^{a,b}=0$ if only one of $a,b, a+b$ is $0$. Moreover, we have a decomposition of motives over $\Q(\z_N)$ with $\Q(\zeta_N)$-coefficients
\begin{align*}
h^1(X_{N, \Q(\z_N)}) \otimes \Q(\z_N) \simeq \bigoplus_{(a, b) \in I_N} X_N^{a, b}.
\end{align*}
\end{prop}

It will be useful to work with motives over $\Q$. 

\begin{dfn}
For $(a, b) \in G_N$, let $[a, b]$ denote its $H_N:=(\Z/N\Z)^*$-orbit.  For a class $[a, b]$, define 
a correspondence of $X_{N, \Q(\z_N)}$ with $\Q$-coefficients by 
$$p_N^{[a, b]}= \sum_{(c, d) \in [a, b]} p_N^{c, d}=\frac{1}{N^2} \sum_{(r,s)\in G_N} \mathrm{Tr}^{\Q(\zeta_N)}_{\Q}(\zeta_N^{-ar-bs}) \Gamma_{g_N^{r, s}}. $$
\end{dfn}
As a direct consequence of \eqref{projector1}, we have 
\begin{align*} 
 \sum_{[a, b] \in H_N \backslash G_N} p_N^{[a, b]}=1 \qquad \text{ and } \qquad p_N^{[a, b]} p_N^{[c, d]}=
  \left\{
\begin{array}{ll}
p_N^{[a, b]}& \text{ if } [a, b]=[c, d];\\
0 & \text{ otherwise.} 
\end{array}
\right.
\end{align*}
We may therefore define a motive $X_N^{[a, b]}:=(X_N, p_N^{[a, b]})$ over $\Q$ with $\Q$-coefficients. The fact that its field of definition is $\Q$ is explained in detail in \cite[p. 40]{Ots11}. As a consequence of Proposition \ref{lem:mot_decomp} (see \cite[Proposition 2.11]{Ots11} for details), we have a decomposition of motives over $\Q$ with $\Q$-coefficients
\begin{equation} \label{decomp h1XN}
h^1(X_{N}) \simeq \bigoplus_{[a, b] \in H_N \backslash I_N} X_N^{[a, b]}.  
\end{equation}

For integers $0<a,b<N$, we define differential $1$-forms on $X_N$ by
\begin{equation}\label{omegars}
    \omega_N^{a,b}:=x^a y^{b-N} \frac{dx}{x}=-x^{a-N} y^b \frac{dy}y.
\end{equation}
If we put $\omega_N^{a,b}:=\omega_N^{\langle a \rangle, \langle b \rangle}$ for $(a,b)\in G_N$, then we have 
\begin{align}
\label{H1dR}    H^1_{\mathrm{dR}}(X_N(\C)) & = \langle \omega_N^{a,b} \mid (a,b) \in I_N \rangle, \\
\label{Omega1}  H^{1,0}(X_N(\C)) & =\langle \omega_N^{a,b} \mid (a,b) \in I_N, \langle a \rangle + \langle b \rangle < N  \rangle.
\end{align}

\begin{lem}\label{lem:proj_1form}
    Given $(a,b), (c,d)\in I_N$, we have 
    \[
    p_N^{a,b} \omega_N^{c,d}=
    \begin{cases}
            \omega_N^{c,d}, & \text{ if } (a,b) = (c,d) \\
            0, & \text{ if } (a,b)\neq (c,d).
        \end{cases} 
    \]
\end{lem}

\begin{proof}
    We have 
    \begin{align*}
        p_N^{a,b} \omega_N^{c,d} &= \dfrac1{N^2} \sum_{(r, s) \in G_N} \zeta_N^{-ar-bs} (g_N^{r, s})^*(\omega_N^{c,d})  = \dfrac1{N^2} \sum_{(r, s) \in G_N} \zeta_N^{-ar-bs} \z_N^{rc+sd} \omega_N^{c,d} \\
        & = \left( \dfrac1{N} \sum_{r \in \Z/N\Z} (\zeta_N^{c-a})^r \right)\left( \dfrac1{N} \sum_{s \in \Z/N\Z} (\zeta_N^{d-b})^s \right) \omega_N^{c,d},
    \end{align*}
    from which the result follows.
\end{proof}

The de Rham cohomology of a Fermat motive $X_N^{[a, b]}$ is defined by 
$$H^1_{\rm dR}(X_N^{[a, b]}) := (p_N^{[a, b]})^* H^1_{\rm dR}(X_N(\C)). $$
By Lemma \ref{lem:proj_1form}, a basis of this vector space is given by 
\begin{equation}\label{basiseq}
    H^1_{\rm dR}(X_N^{[a, b]}) = \langle \omega_N^{c, d} \mid (c, d) \in [a, b] \rangle.
\end{equation}

\begin{cor}\label{cor:rank}
    For all $(a,b)\in I_N$, we have:
    \begin{itemize}
        \item $X^{a,b}_N$ is a rank $1$ motive over $\Q(\zeta_N)$ with coefficients in $\Q(\zeta_N)$;
        \item $X^{[a,b]}_N$ is a rank $\varphi(N)$ motive over $\Q$ with $\Q$-coefficients.
    \end{itemize}
\end{cor}

\begin{proof}
    Combine Proposition \ref{lem:mot_decomp} with the explicit description \eqref{H1dR} and Lemma \ref{lem:proj_1form}.
\end{proof}

By Corollary \ref{cor:rank}, the rank $2$ Fermat motives can be listed as follows: 
\begin{equation}\label{list:rank2}
X_3^{[1, 1]}, \quad X_4^{[1, 1]}, \quad X_4^{[1, 2]}, \quad X_6^{[1, 1]},\quad X_6^{[1, 2]}, \quad X_6^{[1, 3]}, \quad X_6^{[1, 4]}, \quad X_6^{[2, 3]}.
\end{equation}
This list agrees with \cite[Section 4.1]{Ots15} and we will return to it in Section \ref{sec:num}.

The $K_4^{(3)}$ group of a Fermat motive $X_N^{[a,b]}$ is defined by 
\begin{equation}\label{def:K4mot}
K_4^{(3)}(X^{[a,b]}_{N}):=(p_N^{[a,b]})^*K_4^{(3)}(X_{N}).
\end{equation}

\begin{lem}\label{lem:K4decomp}
    We have a direct sum decomposition
\[ 
K_4^{(3)}(X_{N}) \simeq \bigoplus_{[a, b] \in H_N \backslash I_N} K_4^{(3)}(X_{N}^{[a,b]}).
\]
\end{lem}

\begin{proof}
    By \eqref{decomp h0h1h2} and \eqref{decomp h1XN}, there is a decomposition
\[ 
K_4^{(3)}(X_{N}) \simeq K_4^{(3)}(\mathbb{P}^1) \oplus \bigoplus_{[a, b] \in H_N \backslash I_N} K_4^{(3)}(X_{N}^{[a,b]}).
\]
The projective bundle formula \cite[\S 8 Theorem 2.1]{Qui} gives 
\begin{equation}\label{eq:K4P1}
    K_4(\mathbb{P}^1)\otimes \Q=(K_4(\Q)\otimes \Q)^{\oplus 2}=0,
\end{equation}
the last equality following from Borel's theorem \cite[Theorem IV.1.18]{Wei13}.
\end{proof}

\subsection{Fermat quotients}

For integers $0<r,s<N$, let $C_N^{r,s}$ be the smooth projective curve birational to the affine curve 
$$v^N = u^r (1-u)^s.$$ 
The map of affine curves 
\begin{equation}\label{map:quot_rs}
(x,y)\longmapsto (u,v)=(x^N,x^r y^s)
\end{equation} 
induces a morphism $\pi^{r,s}_N \colon X_N \to C_N^{r,s}$ of curves defined over $\Q$.  
When $\gcd(N,r,s)=1$, the curve $C^{r,s}_N$ is called a {\it cyclic Fermat quotient} with genus given by the formula
$$g(C_N^{r,s}) = (N+2-(\gcd(N,r)+\gcd(N,s)+\gcd(N,r+s)))/2.$$
The reason for the name is that the morphism $\pi^{r,s}_N \colon X_N \to C_N^{r,s}$ is a cyclic Galois covering map with 
\[
G_N^{r,s} := \mathrm{Gal}(X_N/C_N^{r,s}) =\{ g_N^{a,b} \mid ar+bs =0 \} = \langle g_N^{s,-r} \rangle \simeq \Z/N\Z.
\] 
With this description in hand, it is not difficult to verify that 
\[
H^{1,0}(C_N^{r,s}(\C))=H^{1,0}(X_N(\C))^{G_N^{r,s}}=\langle \omega_N^{a,b} \mid (a,b) \in I_N, \langle a \rangle + \langle b \rangle < N, br \equiv as \mod N  \rangle.
\]
Moreover, if there exists $t\in H_N=(\Z/N\Z)^\times$ such that 
\[
\{ r,s,N-r-s \} \equiv \{ tr',ts',t(N-r'-s') \} \mod N,
\]
written $(r,s) \sim_N (r',s')$, then the curves $C_N^{r,s}$ and $C_N^{r',s'}$ are isomorphic over $\Q$. 

\begin{rmk}
It is known that a cyclic Fermat quotient is hyperelliptic if, and only if, $(r,s)\sim_N (1,1)$ or $N=2n$ and $(r,s)\sim_N (1,n)$. The hyperelliptic involution of $C^{1,1}_N$ is given by $(u,v)\mapsto (1-u, v)$ and is induced by the involution $\iota$ \eqref{sym:1} via the quotient map $\pi^{1,1}_N$.
\end{rmk}

\begin{rmk}\label{rmk:Cmot}
    In the case where $N$ is prime and $(a,b)\in I_N$, $C_N^{a,b}:=C_N^{\langle a\rangle,\langle b\rangle}$ is a cyclic Fermat quotient of genus $(N-1)/2$. We then have
$
H^1_{\mathrm{dR}}(C_N^{a,b})=\langle \omega_N^{ta,tb} \mid t\in H_N \rangle=\langle \omega_N^{r,s} \mid (r,s)\in [a,b] \rangle.
$
Following \cite[Remark 2.13]{Ots11}, it is possible to show that $\pi^{r,s}_N$ induces an isomorphism of motives over $\Q$ with $\Q$-coefficients
$
X_N^{[a,b]}\simeq h^1(C_N^{a,b}).
$
\end{rmk}

\section{Polylogarithmic complexes of curves} \label{sec:cocycles}

\subsection{Polylogarithmic functions}

We briefly recall the salient features of the relevant polylogarithmic functions and refer the reader to \cite{Goncharov} for details. 

\subsubsection{Simple polylogarithms}
For $k\in \Z_{\geq 1}$, the $k$-th polylogarithmic function is defined for complex $z$ with $\vert z\vert < 1$ by the series
\begin{equation}\label{def:Lik}
    \mathrm{Li}_k(z):=\sum_{n\geq 1} \frac{z^n}{n^k}.
\end{equation}
Note that 
\begin{equation}\label{polylog:deriv}
    \mathrm{Li}_1(z)=-\log(1-z) \qquad \text{ and } \qquad \frac{d}{dz} \mathrm{Li}_k(z)=\frac{\mathrm{Li}_{k-1}(z)}{z}dz.
\end{equation}
The inductive formula 
\[
\mathrm{Li}_k(z)=\int_0^z \frac{\mathrm{Li}_{k-1}(t)}{t}dt
\]
gives the analytic continuation of $\mathrm{Li}_k$ to a multi-valued function on $\mathbb{P}^1(\C)\setminus \{ 0,1,\infty \}$. We have the special values 
\[
\mathrm{Li}_k(0)=0 \qquad \text{ and } \qquad \mathrm{Li}_k(1)=\zeta(k), \quad k\geq 2.
\]

\subsubsection{Bloch--Wigner's dilogarithm}
The Bloch--Wigner dilogarithm $D_2 \colon \mathbb{P}^1(\C) \to \R$ is defined by
\begin{equation}\label{BWlog}
    D_2(z)=
    \begin{cases}
        \Im (\mathrm{Li}_2(z) + \log(1-z)\log \vert z\vert), & \text{ if } \vert z\vert \leq 1; \\
        -D_2(1/z), & \text{ if } \vert z\vert \geq 1.
    \end{cases}
\end{equation}
It is single-valued, real-analytic on $\mathbb{P}^1(\C)\setminus \{ 0,1,\infty \}$ and continuous at $0,1,\infty$ with values 
\begin{equation}\label{D2:value}
    D_2(0)=D_2(1)=D_2(\infty)=0.
\end{equation}
It satisfies a remarkable $5$-term functional equation that we now recall. If $z_1, \ldots, z_4 \in \mathbb{P}^1(\C)$ are $4$ distinct points with coordinates $\tilde z_1, \ldots, \tilde z_4 \in \C\cup \{\infty \}$, we let
\[
r(z_1, \ldots, z_4):=\frac{(\tilde z_1-\tilde z_4)(\tilde z_2-\tilde z_3)}{(\tilde z_1-\tilde z_3)(\tilde z_2-\tilde z_4)} \in \C\setminus \{ 0,1 \}
\]
denote the cross-ratio of their coordinates (with the convention that $\frac{\infty}{\infty} =1$). For any five distinct points $z_0, \ldots, z_4 \in \mathbb{P}^1(\C)$, define 
\begin{equation}\label{R2}
R_2(z_0, \ldots, z_4):=\sum_{i=0}^4 (-1)^i \{ r(z_0, \ldots, \hat{z}_i, \ldots, z_4) \} \in \Z[\C\setminus \{ 0,1 \}],
\end{equation}
where we write $\{ z \} \in \Z[\C\setminus \{ 0,1 \}]$ for the basis element corresponding to $z\in \C\setminus \{ 0,1 \}$.
Viewing $D_2$ as a homomorphism $\Z[\C\setminus \{ 0,1 \}] \to \R$, we then have 
\begin{equation}\label{FE:D2}
D_2(R_2(z_0, \ldots, z_4))=0, \quad \text{ for all distinct } z_0, \ldots, z_4 \in \mathbb{P}^1(\C).
\end{equation}
Note that any $5$-tuple $(z_0, \ldots, z_4)$ of distinct points is projectively equivalent to $(0, \infty, 1, x, y)$ for some $x\neq y \in \C$. In particular, we have 
\begin{equation}\label{FE:D22}
D_2\left( \{ x \}- \{ y \}+\left\{\frac{y}{x}\right\}-\left\{\frac{1-x^{-1}}{1-y^{-1}}\right\}+\left\{\frac{1-x}{1-y}\right\} \right)=0.
\end{equation}
The Bloch--Wigner function further satisfies the relation 
\begin{equation}\label{FE:D23}
    D_2(z)=-D_2(1-z)=-D_2(z^{-1}),
\end{equation}
which can be viewed as a degenerate case of \eqref{FE:D2}.

\subsubsection{Zagier's trilogarithm}
A generalization of the Bloch--Wigner dilogarithm is Zagier's trilogarithm \cite{Zagier}
\begin{equation}\label{GonchL3}
\mathcal{L}_3(z):=\Re\Bigl(\mathrm{Li}_3(z)-\mathrm{Li}_2(z)\log(\vert z\vert)+\frac{1}{3}\mathrm{Li}_1(z)\log(\vert z\vert)^2\Bigr)
\end{equation}
whose definition can be found in \cite[(1.3)]{Goncharov}. The function $\mathcal{L}_3$ is single-valued, real-analytic on $\mathbb{P}^1(\C)\setminus \{ 0,1,\infty \}$ and continuous at $0,1,\infty$ with values 
\[
\mathcal{L}_3(0)=\mathcal{L}_3(\infty)=0 \qquad \text{ and } \qquad \mathcal{L}_3(1)=\zeta(3).
\]
It satisfies remarkable functional equations \cite[Theorem 1.3]{Goncharov} whose explicit statement we will not need. One advantage of this version of the trilogarithm is that its functional equations do not involve any remainder terms. 

\subsection{Polylogarithmic complexes and de Jeu's map}\label{s:compl}

As mentioned in the introduction, both Goncharov \cite{Goncharov} and de Jeu \cite{Jeu95,Jeu96} have constructed polylogarithmic complexes that are expected to compute Adams eigenspaces of the $K$-theory of fields. In our situation, and for any field $F$ of characteristic zero, de Jeu has defined a complex $\widetilde{\mathcal{M}}_{(3)}^{\bullet} (F)$ using multi-relative $K$-theory and has constructed a map $H^2(\widetilde{\mathcal{M}}_{(3)}^{\bullet} (F)) \to K_4^{(3)}(F)$ \cite{Jeu95}. On the other hand, Goncharov \cite{Goncharov} has constructed a weight $3$ complex $\Gamma(F,3)$ \footnote{There exist different notations for this complex. It is denoted $B_F(3)$ in \cite[p. 219]{Goncharov}, $\Gamma'(F,3)$ in \cite{Jeu00}, and $\Gamma(F,3)$ in \cite{Bru20}. We shall use the latter notation.} defined below (Definition \ref{def:complex}), and it follows from results in \cite[Section 5]{Jeu00} that there is a map $H^2(\Gamma(F,3))\to H^2(\widetilde{\mathcal{M}}_{(3)}^{\bullet} (F))$.
A concise survey of these complexes and de Jeu's map can be found in \cite[Section 4]{Bru20}. Here, we content ourselves with the minimal amount of detail necessary for our purposes. 

\subsubsection{The case of fields}

Let $F$ be a field of characteristic $0$. The building blocks of Goncharov's polylogarithmic complexes are certain $\Q$-vector spaces $B_n(F)$ defined, for $n\in \Z_{\geq 1}$, as quotients of $\Q[\mathbb{P}^1(F)]$ by a subspace $R_n(F)$ of relations that mimick the functional equations satisfied by the $n$-th polylogarithm $\mathrm{Li}_n$ \eqref{def:Lik}. The image of an element $f\in \mathbb{P}^1(F)$ in $B_n(F)$ will be denoted $\{ f \}_n$. 

We only require the case $n=2$. Explicitly, the vector space $B_2(F)$ can be described as the quotient of $\Q[F\setminus \{ 0,1 \}]$ by the subspace $R_2(F)$ generated by elements of the form 
$
R_2(z_0, \ldots, z_4)
$
\eqref{R2} with distinct points $z_0, \ldots, z_4 \in \mathbb{P}^1(F)$. These relations reflect the functional equations \eqref{FE:D2} of the Bloch--Wigner dilogarithm $D_2$ \eqref{BWlog}. In particular, in accordance with \eqref{FE:D22}, \eqref{FE:D23}, and \eqref{D2:value}, we have the following equalities in $B_2(F)$ for all $f,g \in F\setminus \{ 0,1 \}$ with $f\neq g$:
\begin{equation}\label{relationsB2}
\begin{cases}
    \{ f \}_2- \{ g \}_2+\left\{\frac{g}{f}\right\}_2-\left\{\frac{1-f^{-1}}{1-g^{-1}}\right\}_2+\left\{\frac{1-f}{1-g}\right\}_2=0 & \\
    \{1-f\}_2=\{f^{-1}\}_2=-\{f\}_2 & \\
\end{cases}
\end{equation}
By convention, we put $\{ 0 \}_2=\{ 1 \}_2=\{ \infty \}_2=0$ in $B_2(F)$. We then define a map
\begin{equation*}
\begin{array}{cccc}
    \tilde{\delta}_F \colon & \Q[F] & \longrightarrow & \Lambda^2 F_{\Q}^\times  \\
     & \{ f \} &  \longmapsto & \begin{cases} (1-f)\wedge f & \textrm{if } f \neq 0, 1 \\
     0 & \textrm{otherwise}
     \end{cases}
\end{array}
\end{equation*}
One can check that $\tilde{\delta}_F$ induces a map $\delta_F \colon B_2(F) \to \Lambda^2 F_{\Q}^\times$.

\begin{dfn}
    The kernel of $\delta_F$ is called the Bloch group of $F$.
\end{dfn}

By Suslin's work \cite{Sus90}, the Bloch group of $F$ is isomophic to $K_3^{(2)}(F)$.
A useful relation in $B_2(F)$ is the so-called distribution relation.

\begin{lem}\label{rel2}
    Let $d \geqslant 1$ be an integer. If $F$ contains a primitive $d$-th root of unity $\z_d$, then the equality
            \begin{equation*}
                \sum_{k \in \Z/d\Z} \{\zeta_d^k f\}_2 = \dfrac1{d} \{f^d\}_2 
            \end{equation*}
        holds in $B_2(F)$ for all $f\in F$.
\end{lem}

\begin{proof}
There is a variant of $B_2(F)$ where the subspace of relations $R_2(F)$ is replaced by another (less explicit) group $\tilde{\mathcal{R}}_2(F)$, defined in \cite[Remark, p.~226]{Goncharov}. The latter group is generated by the elements $A\rvert_{t=1} - A\rvert_{t=0}$ for any $A$ in $\ker(\tilde{\delta}_{F(t)})$. It follows from Suslin's work \cite{Sus90} and an argument of de Jeu \cite[Remark 5.3]{Jeu00} that $R_2(F) = \tilde{\mathcal{R}}_2(F)$.

Consider now the formal sum
\begin{equation*}
    A = \frac{1}{d} \{(ft)^d\} - \sum_{k \in \Z/d\Z} \{\zeta_d^k ft\} \in \Q[F(t)].
\end{equation*}
Then
\begin{align*}
    \tilde{\delta}_{F(t)}(A) & = \frac{1}{d} (ft)^d \wedge (1-(ft)^d) - \sum_{k \in \Z/d\Z} (\zeta_d^k ft) \wedge (1-\zeta_d^k ft) \\
    & = (ft) \wedge (1-(ft)^d) - \sum_{k \in \Z/d\Z} (ft) \wedge (1-\zeta_d^k ft) \\
    & = (ft) \wedge (1-(ft)^d) - (ft) \wedge \prod_{k \in \Z/d\Z} (1-\zeta_d^k ft) \\
    & = 0.
\end{align*}
So by definition, $R_2(F) = \tilde{\mathcal{R}}_2(F)$ contains
\begin{equation*}
    A \rvert_{t=1} - A \rvert_{t=0} = \frac{1}{d} \{f^d\}_2 - \sum_{k \in \Z/d\Z} \{\zeta_d^k f\}_2. \qedhere
\end{equation*}
\end{proof}

\begin{rmk}
    The case $d=2$ of Lemma \ref{rel2} can also be derived directly from the $5$-term relation \eqref{relationsB2} by taking $(f,g)=(f, f^{-1})$.  
\end{rmk}

The group $B_3(F)$ can be defined as a quotient of $\Q[F\setminus \{ 0,1 \}]$ by a subspace $R_3(F)$ of relations that reflect the functional equations satisfied by Zagier's trilogarithm $\mathcal{L}_3$ \eqref{GonchL3}. We will not need the explicit description here, but instead refer the interested reader to \cite[p. 214 \& 218]{Goncharov}.

\begin{dfn}\label{def:complex}
    Goncharov's weight $3$ polylogarithmic complex, placed in degrees $1$ to $3$, is 
\[
\begin{array}{cccccc}
    \Gamma(F,3) \colon & B_{3}(F) & \longrightarrow & B_{2}(F) \otimes F_{\Q}^\times & \overset{\partial=\delta_F\wedge \mathrm{id}}{\longrightarrow} & \Lambda^3 F_{\Q}^\times  \\
     & \{ f \}_3 &  \longmapsto & \{ f \}_2 \otimes f & &  \\
     & & & \{ f \}_2 \otimes g & \longmapsto & (1-f)\wedge f \wedge g.
\end{array}
\]
\end{dfn}

Goncharov \cite[Conjecture A and
Conjecture 1.17, p. 222–223]{Goncharov} conjectures that $H^2(\Gamma(F,3))$ is canonically isomorphic to $K_4^{(3)}(F)$. As mentioned in the opening paragraph of Section \ref{s:compl}, de Jeu's results in \cite{Jeu95,Jeu96} provide a map that is a candidate for this conjectural isomorphism: 
\begin{equation}\label{map:deJeu}
    \varphi_{\mathrm{dJ}} \colon H^2(\Gamma(F,3)) \longrightarrow K_4^{(3)}(F).
\end{equation}
This map is canonical up to the choice a universal sign. Making such a choice, the map \eqref{map:deJeu} is functorial in $F$.

\subsubsection{The case of curves}

Let $X$ be a smooth projective geometrically connected curve over a number field $k$ with function field $F=k(X)$. If $p$ is any closed point of $X$, then the residue map at $p$ is defined as
   \begin{equation}\label{map:res}
\begin{array}{cccc}
    \mathrm{Res}_p \colon & B_{2}(F) \otimes F_{\Q}^\times & \longrightarrow & B_2(k(p))  \\
     & \{ f \}_2 \otimes g &  \longmapsto & \ord_p(g) \{ f(p) \}_2,
\end{array}
\end{equation}
where we recall the convention that $\{ 0 \}_2=\{ 1 \}_2=\{ \infty \}_2=0$ in $B_2(k(p))$. The latter convention implies that $\mathrm{Res}_p$ is trivial on the image of $B_3(F)$, and one readily checks that the residue map induces a map
\begin{equation}\label{map:res_cohom}
    \mathrm{Res}_p \colon H^2(\Gamma(F,3)) \longrightarrow \ker(\delta_{k(p)})
\end{equation}
to the Bloch group of $k(p)$.

Given a finite set of closed points $S$, we set $U:=X\setminus S$. Soulé's localization sequence in $K$-theory with weights \cite[Remarques, p. 525]{Sou85} then yields an exact sequence 
\begin{equation*}
0 \longrightarrow K_4^{(3)}(U) \longrightarrow K_4^{(3)}(F) \overset{\sum r_p}{\longrightarrow} \bigoplus_{p\in U} K_3^{(2)}(k(p)),
\end{equation*}
where $r_p$ denotes the $K$-theoretic residue map at $p$. This allows us to view $K_4^{(3)}(U)$ as a subgroup of $K_4^{(3)}(F)$. The following result will be used to construct elements of $K_4^{(3)}(U)$. 

\begin{thm}[Theorem 4.5 of \cite{Bru20}]\label{thm:K4Y}
    Let $\xi=\sum_i n_i \{ f_i \}_2 \otimes g_i$ be a $2$-cocycle in the Goncharov complex $\Gamma(F, 3)$, with $f_i, g_i \in F^\times$ and $n_i \in \Q$. Assume
that all the functions $f_i, 1-f_i,$ and $g_i$ belong to $\mathcal{O}(U)^\times$. Then the image of $\xi$ under de Jeu’s
map \eqref{map:deJeu} belongs to $K_4^{(3)}(U)$.
\end{thm}

The next result allows us to extend elements of $K_4^{(3)}(U)$ to elements of $K_4^{(3)}(X)$ under certain assumptions.

\begin{thm}\label{rmk:K4X}
    Suppose that $k$ is totally real.
    Let $\xi=\sum_i n_i \{ f_i \}_2 \otimes g_i$ be a $2$-cocycle in the Goncharov complex $\Gamma(F, 3)$, with $f_i, g_i \in F^\times$ and $n_i \in \Q$. Assume
    that all the functions $f_i, 1-f_i,$ and $g_i$ belong to $\mathcal{O}(U)^\times$. Assume moreover that $\mathrm{Res}_p(\xi)=0$ for all $p\in S$, where $\mathrm{Res}_p$ is the map \eqref{map:res}. Then the image of $\xi$ under de Jeu’s map \eqref{map:deJeu} belongs to $K_4^{(3)}(X)$.
\end{thm}

\begin{proof}
    As in \cite[(29)]{Bru20}, we have the localization exact sequence in $K$-theory
\begin{equation}\label{ses:KXU}
    0 \longrightarrow K_4^{(3)}(X) \longrightarrow K_4^{(3)}(U) \overset{\sum r_p}{\longrightarrow} \bigoplus_{p\in S} K_3^{(2)}(k(p)).
\end{equation}
In the case when $k$ is totally real, de Jeu proved that, upon identifying $\ker(\delta_{k(p)})$ and $K_3^{(2)}(k(p))$, his map \eqref{map:deJeu} commutes with taking the residue maps $2\mathrm{Res}_p$ and $r_p$ \cite[Theorem 4.4]{Bru20} (up to sign).
The result then follows by Theorem \ref{thm:K4Y}, exactness of the sequence \eqref{ses:KXU}, and de Jeu's compatibility of residue maps.
\end{proof}

\section{Weight $3$ Goncharov $2$-cocycles of Fermat curves}\label{s:2coc}

We carry out the strategy explained in Section \ref{s:compl} in the case of Fermat curves and construct explicit elements in $K_4^{(3)}(\Q(X_N))$.

\subsection{The construction}
Let $\z$ be a root of unity and let $\Q_{\z}:=\Q(\z)$. Write $X_{N,\Q_\z}$ for the base-change of $X_N$, and let $F_N(\z)=\Q_{\z}(X_{N,\Q_{\zeta}})$ be its function field.
For integers $r, s\in \Z$ with $s\neq 0$, let $S_N^{r,s}(\z)$ be the finite subset of closed points of $X_{N,\Q_{\zeta}}$ consisting of the cusps and of the zeros and poles of $1- \z x^ry^s$. 
Put $U_N^{r,s}(\z)=X_N \setminus S_N^{r,s}(\z)$ and define the element 
\[
\xi_N^{r,s}(\z):=\{x^N\}_2 \otimes (1- \z x^r y^s) - \frac{N}{s} \{\z x^r y^s\}_2 \otimes x^N\in B_2(F_N(\z))\otimes F_N(\z)^\times_{\Q}.
\]
When $\zeta=1$, we write $\xi^{r, s}_N$ instead of $\xi^{r, s}_N(1)$.

\begin{prop}
    The element $\xi_N^{r,s}(\z)$ is a $2$-cocycle in the Goncharov complex $\Gamma(F_N(\z), 3)$. 
\end{prop}

\begin{proof}
Writing $u=x^N$ and $v=x^r y^s$, observe that $v^N=(x^r y^s)^N=u^r(1-u)^s$.
We then have 
\begin{align*}
\partial(\xi^{r,s}_N(\z)) & = (1-u)\wedge u \wedge (1- \z v)-  \frac{N}{s}(1- \z v)\wedge (\z v) \wedge u \\
& =(1-u)\wedge u \wedge (1- \z v)-  \frac{N}{s}(1- \z v)\wedge v \wedge u - \frac{N}{s}(1- \z v)\wedge \z \wedge u \\
& =(1-u)\wedge u \wedge (1- \z v)-  \frac{N}{s}(1- \z v)\wedge v \wedge u \\
& =(1-u)\wedge u \wedge (1- \z v)-  \frac{1}{s}(1- \z v)\wedge v^N \wedge u \\
& =(1-u)\wedge u \wedge (1- \z v)-  \frac{1}{s}(1- \z v)\wedge u^r(1-u)^s \wedge u \\
& =(1-u)\wedge u \wedge (1- \z v)-  \frac{1}{s}(1- \z v)\wedge (1-u)^s \wedge u \\
& =(1-u)\wedge u \wedge (1- \z v)-  (1- \z v)\wedge (1-u) \wedge u =0,
\end{align*}
where we used the fact that the term $\frac{N}{s}(1- \z v)\wedge \z \wedge u$ is torsion and can thus be ignored.    
\end{proof}

\begin{cor}\label{cor:K4open}
    The image of $\xi_N^{r,s}(\z)$ under de Jeu's map \eqref{map:deJeu} defines an element $$\Xi_N^{r,s}(\zeta):=\varphi_{\mathrm{dJ}}(\xi_N^{r,s}(\z)) \in K_4^{(3)}(U_N^{r,s}(\z)).$$
\end{cor}

\begin{proof}
    This is a direct application of Theorem \ref{thm:K4Y}.
\end{proof}

\begin{rmk}
    When $0<r,s<N$, the element $\xi_N^{r,s}(\z)$ is obtained by pulling back the cocycle 
    $$\{ u \}_2 \otimes (1-\z v) - \frac{N}{s} \{ \z v \}_2 \otimes u \in B_2(\Q_{\z}(C^{r,s}_N))\otimes \Q_{\z}(C^{r,s}_N)^\times_{\Q}$$ 
    by the quotient map $\pi_N^{r,s}$ \eqref{map:quot_rs}.
\end{rmk}

\begin{rmk}
    When $r=0$, the element $\xi_N^{0,s}(\zeta)$ comes from $\mathbb{P}^1$. 
When $s=0$, the element $\xi_N^{r,0}(\zeta)$ is not defined.  
If $r+s=0$, the element $\xi_N^{r,s}(\zeta)$ is obtained by pulling back an element of $C_N^{r,s} \simeq \mathbb{P}^1$. 
We will thus assume that $r, s, r+s \neq 0$. 
\end{rmk}

The elements $\Xi^{r,s}_N(\z)$ satisfy certain distribution relations with respect to the degree of the Fermat curve, as stated in the next proposition.

\begin{prop}\label{dist:cocycle}
Let $d$ be a positive divisor of $N$. For any primitive $d$-th root of unity $\z_d$, we have the following equality in $B_2(F_N(\z,\z_d))\otimes F_N(\z,\z_d)^\times_{\Q}$:
    \begin{equation*}
    \sum_{k \in \Z/d\Z} \xi_N^{r,s}(\zeta \zeta_d^k) = \pi_{N,N/d}^*(\xi_{N/d}^{r,s}(\zeta^d)),
    \end{equation*}
where the map $\pi_{N, N/d}$ is defined in \eqref{map:divideN}. As a consequence, we have in $K_4^{(3)}(F_N(\zeta,\zeta_d))$:
\begin{equation*}
    \sum_{k \in \Z/d\Z} \Xi_N^{r,s}(\zeta \zeta_d^k) = \pi_{N,N/d}^*(\Xi_{N/d}^{r,s}(\zeta^d)).
\end{equation*}
\end{prop}

\begin{proof}
We have
    \begin{align*}
        \sum_{k \in \Z/d\Z} \xi_N^{r,s}(\zeta \zeta_d^k) &= \{x^N\}_2 \otimes \prod_{k\in \Z/d\Z}(1- \z \z_d^k x^r y^s) - \frac{N}{s} \left( \sum_{k \in \Z/d\Z} \{\z\z_d^k x^r y^s\}_2 \right) \otimes x^N \\
        & = \{x^N\}_2 \otimes (1- \z^d x^{dr} y^{ds}) - \frac{N}{s} \left( \sum_{k \in \Z/d\Z} \{\z_d^k\z x^r y^s\}_2 \right) \otimes x^N \\
        & =\{x^N\}_2 \otimes (1- \z^d x^{dr} y^{ds}) - \frac{N}{ds} \{\z^d x^{dr} y^{ds}\}_2 \otimes x^N \\
        &= \pi_{N,N/d}^*(\xi_{N/d}^{r,s}(\zeta^d)),
    \end{align*}
    where in the second-to-last equality we used Lemma \ref{rel2}. The last part of the proposition follows from the functoriality of de Jeu's map \eqref{map:deJeu}.
\end{proof}

\subsection{Action of $G_N$}

\begin{prop}\label{prop:trace}
Consider integers $0<r,s<N$ such that $\gcd(N, r,s)=1$ and $\tilde r, \tilde s \in \Z$ such that $\tilde r, \tilde s, \tilde r+ \tilde s \neq 0$. The following statements hold:
\begin{enumerate}
\item Given $(a,b)\in G_N$, we have $g_N^{a,b}\xi_N^{\tilde r, \tilde s}(\z)=\xi^{\tilde r, \tilde s}_N(\z \z_N^{a \tilde r + b \tilde s})$ in $B_2(F_N(\z,\z_N))\otimes F_N(\z,\z_N)^\times_{\Q}$;
\item If $s \tilde r \equiv r \tilde s \mod N$, then $\xi^{\tilde r, \tilde s}_N(\z)\in B_2(F_N(\z))\otimes F_N(\z)^\times_{\Q}$ is invariant under the action of the group $G_N^{r,s}$.
\end{enumerate}
\end{prop}

\begin{proof}
The first point is easy to check. Any element of $G_N^{r,s}$ is of the form $g_N^{sk,-rk}$ for some $k\in \Z/N\Z$. By (i), we see that 
$
g_N^{sk,-rk}\xi_N^{\tilde r, \tilde s}(\z)=\xi_N^{\tilde r, \tilde s}(\z\z^{k(s\tilde r -r \tilde s)}),
$
from which (ii) follows.
\end{proof}

\begin{prop} \label{projection}
Given integers $r,s\in \Z$ with $s\neq 0$, we have 
\begin{align*}
\xi^{r, s}_N(\zeta)= \sum_{\substack{(a,b) \in G_N \\ br \equiv as \bmod{N}}} p_N^{a,b} \xi^{r, s}_N(\zeta) \in B_2(F_N(\z))\otimes F_N(\z)^\times_{\Q}. 
\end{align*}
\end{prop}
\begin{proof}
By Proposition \ref{prop:trace}(i), for all $k\in \Z/N\Z$, we have
\begin{equation*}
g_N^{sk,-rk} \xi^{r, s}_N(\zeta) = \xi^{r, s}_N(\zeta).
\end{equation*}
Moreover, we have the relation $\Gamma_{g_N^{\alpha,\beta}} p_N^{a,b} = \zeta_N^{a\alpha+b\beta} p_N^{a,b}$ in the group $\mathrm{Corr}^0(X_N,X_N)_{\Q(\z_N)}$ of correspondences with $\Q(\zeta_N)$-coefficients. 
It follows that in the group $B_2(F_N(\z))\otimes F_N(\z)^\times_{\Q}$, we have
\begin{align*}
\xi^{r, s}_N(\zeta) & = \frac1N\left(\sum_{k \in \Z/N\Z} g_N^{sk, -rk}\right) \xi^{r, s}_N(\zeta)
 \overset{\eqref{projector1}}{=} \frac1N\left(\sum_{k \in \Z/N\Z} g_N^{sk, -rk}\right) \sum_{(a, b) \in G_N} p_N^{a, b} \xi^{r, s}_N(\zeta) \\
& = \frac1N \sum_{k \in \Z/N\Z} \sum_{(a,b) \in G_N} \zeta_N^{ask-brk} p_N^{a,b} \xi^{r, s}_N(\zeta)  
= \frac1N  \sum_{(a,b) \in G_N} \left( \sum_{k \in \Z/N\Z} (\zeta_N^{as-br})^k \right) p_N^{a,b} \xi^{r, s}_N(\zeta)  \\
& = \frac1N  \sum_{(a,b) \in G_N} \left( N\delta_{br,as} \right) p_N^{a,b} \xi^{r, s}_N(\zeta)  
= \sum_{\substack{(a,b) \in G_N \\ br \equiv as \bmod{N}}} p_N^{a,b} \xi^{r, s}_N(\zeta) .
\end{align*}
\end{proof}

Proposition \ref{projection} allows us to situate the $K_4$ elements $\Xi_N^{r,s}(\z)$ with respect to the decomposition of $K_4^{(3)}(X_{N,\Q_\z})$ established in Lemma \ref{lem:K4decomp}.

\begin{cor}\label{cor:pNrs}
    Let $r,s \in \Z$ with $s\neq 0$. If $\Xi^{r,s}_N(\zeta)\in K_4^{(3)}(X_{N,\Q_{\zeta}})$, then
    \[
    \Xi^{r,s}_N(\zeta)\in  \bigoplus_{\substack{[a,b] \in H_N\backslash I_N \\ br \equiv as \bmod{N}}} K_4^{(3)}(X_{N,\Q_{\zeta}}^{[a,b]}).
    \]
    In particular, if $N$ is prime, then $\Xi^{r,s}_N(\zeta)\in K_4^{(3)}(X_{N,\Q_{\zeta}}^{[r,s]})$.
\end{cor}

\begin{proof}
    Direct consequence of Proposition \ref{projection}, functoriality of de Jeu's map, and the definition \eqref{def:K4mot} of $K_4^{(3)}(X_N^{[a,b]})$ upon observing that $K_4^{(3)}(X_N^{[0,0]})=K_4^{(3)}(\mathbb{P}^1)=0$, the last equality following from \eqref{eq:K4P1}. 
\end{proof}

\subsection{Action of symmetries}

We now analyze the behavior of the elements $\xi^{r, s}_N(\zeta)$ under pull-backs by the symmetries \eqref{sym:1}, \eqref{sym:2}, and \eqref{sym:3}.

\begin{prop} \label{iota}
We have the following equalities of cocycles modulo $2$-coboundaries: 
\begin{enumerate}
\item  $\xi^{r, s}_N(\zeta)=\xi^{-r, -s}_N(\z^{-1})$;
\item $\iota^* \xi^{r, s}_N(\zeta)=-  \xi^{s, r}_N(\zeta)$;
\item $\tau^* \xi^{r, s}_N(\zeta)=-\xi^{-r-s, s}_N(\z \e^s)$;
\item $\sigma^* \xi^{r, s}_N(\zeta) =-\xi^{r, -r-s}_N(\z \e^r).$
\end{enumerate}
\end{prop}

\begin{proof}
During the course of the proof, we will use the relation 
$\{1-f\}_2=\{f^{-1}\}_2=-\{f\}_2$ \eqref{relationsB2}.

For the proof of (i), we observe that $\xi^{r, s}_N(\zeta)-\xi^{-r, -s}_N(\z^{-1})$ is equal to
\begin{multline*}
\{x^N\}_2 \otimes \left(1- \z {x^r}{y^s} \right) + \frac{N}s\left\{\z {x^r}{y^s}\right\}_2 \otimes x^N  - 
\{x^N\}_2 \otimes \left(1- \z^{-1} {x^{-r}}{y^{-s}} \right) +  \frac{N}s\left\{\z^{-1} {x^{-r}}{y^{-s}}\right\}_2 \otimes x^N \\
 = \{x^N \}_2 \otimes \left(- \z {x^r} y^s \right) 
=\{x^N\}_2 \otimes y^s = \{1-y^N\}_2 \otimes y^s 
= -\{y^N\}_2 \otimes y^s, 
\end{multline*}
which is a coboundary. 
For the proof of (ii), we have 
\begin{equation*}
\iota^* \xi^{r, s}_N(\zeta)+  \xi^{s, r}_N(\zeta) = - \frac{N}{s} \{\z x^s y^r\}_2 \otimes y^N - \frac{N}{r} \{\z x^s y^r \}_2 \otimes x^N = - \frac{N}{rs} \{\z x^s y^r\}_2 \otimes (x^{Ns} y^{Nr}), 
\end{equation*}
which is a coboundary. 
For the proof of (iii), we get
\begin{align*}
 \tau^* \xi^{r, s}_N(\zeta) & =  \{x^{-N}\}_2 \otimes (1- \z\varepsilon^s x^{-(r+s)} y^s) - \frac{N}{s} \{\z\varepsilon^s x^{-(r+s)} y^s\}_2 \otimes x^{-N} \\
 & =  -\{x^{N}\}_2 \otimes (1- \z\varepsilon^s x^{-(r+s)} y^s) + \frac{N}{s} \{\z\varepsilon^s x^{-(r+s)} y^s\}_2 \otimes x^{N} \\
 & = - \xi^{-r-s, s}_N(\z \e^s).
\end{align*}
For the proof of (iv), since $\sigma=\iota \tau \iota$, we see that
\begin{equation*}
 \sigma^* \xi^{r, s}_N(\zeta) =\iota^* \tau^* (-  \xi^{s, r}_N(\zeta))=\iota^*(\xi^{-r-s, r}_N(\z \e^r))=-\xi^{r, -r-s}_N(\z \e^r).
\end{equation*}
\end{proof} 

\section{Elements in $K_4^{(3)}$ of Fermat curves} \label{sec:elements}
We now specialize to a case of particular interest when $(r,s)=(-2,1)$ and $\zeta=1$. In this case, we can use Theorem \ref{rmk:K4X} to show that the corresponding elements $\Xi_N^{-2,1}$ defined in Section \ref{s:2coc} extend to $K_4^{(3)}(X_N)$. We begin by restating Proposition \ref{iota} in this case.

\begin{prop}\label{cor:pb}
We have the following equalities modulo $2$-coboundaries: 
\begin{enumerate}
\item  $\xi^{-2, 1}_N=\xi^{2, -1}_N$;
\item $\iota^* \xi^{-2, 1}_N=-  \xi^{1, -2}_N$;
\item $\tau^* \xi^{-2, 1}_N=-\xi^{1, 1}_N(\e)$;
\item $\sigma^* \xi^{-2, 1}_N =-\xi^{-2, 1}_N(\e^{-2});$
\item If $0< r,s<N$ are integers such that $\gcd(N,r,s)=1$ and $r\equiv -2s \mod N$, then $\xi^{-2, 1}_N$ is invariant under the action of $G_N^{r,s}$. 
\end{enumerate}
In particular, if $N$ is odd we may choose $\varepsilon=-1$, in which case we have the equalities modulo coboundaries $\sigma^* \xi^{-2, 1}_N =-\xi^{-2, 1}_N$ and
\[ 
\tau^* \xi^{-2, 1}_N=- \xi^{1, 1}_N(-1)=\{x^N\}_2 \otimes (1 + x y) - N \{ - x y \}_2 \otimes x^N. 
\]
\end{prop}

\begin{proof}
    This is a special case of Proposition \ref{iota} and Proposition \ref{prop:trace}(ii).
\end{proof}

\begin{thm}\label{prop:res1}
The elements $\Xi_N^{-2, 1}$ and $\Xi_N^{1, -2}$ define elements of $K_4^{(3)}(X_N)$. The elements $\Xi_N^{1, 1}(\e)$ and $\Xi_N^{-2, 1}(\e^{-2})$ define elements of $K_4^{(3)}(X_{N,K})$ with $K=\Q(\e)$.
If $N$ is odd, then $\Xi_N^{1,1}(-1)$ also defines an element of $K_4^{(3)}(X_N)$. 
\end{thm}

\begin{proof}
By Proposition \ref{cor:pb} and the functoriality of de Jeu's map \eqref{map:deJeu}, it suffices to show that
$$\Xi_N^{-2, 1}=\varphi_{\mathrm{dJ}}(\xi_N^{-2, 1}) \in K_4^{(3)}(X_N). $$  
We already know that $\Xi_N^{-2,1}\in K_4^{(3)}(U^{-2,1}_N)$ by Corollary \ref{cor:K4open}. By Theorem \ref{rmk:K4X}, since $X_N$ is defined over the totally real field $\Q$, it suffices to check that the element 
\[
\xi^{-2,1}_N=\{x^N\}_2 \otimes \left(1- \frac{y}{x^2} \right) - N \left\{\dfrac{y}{x^2} \right\}_2 \otimes x^N \in H^2(\Gamma(\Q(X_N),3))
\]
lies in the kernel of the residue map $\mathrm{Res}_p$ \eqref{map:res_cohom} for all $p\in S^{-2,1}_N$.

If $p$ is a cusp, this is true since $\{0\}_2=\{1\}_2=\{\infty\}_2=0$ in $B_2(\Q(p))$. Otherwise, let us write $p = (x_0,y_0)$ with $x_0, y_0 \neq 0$. Then $p$ is a zero of the function $1-x^{-2}y$, \emph{i.e.}~$y_0 = x_0^2$, which implies $x_0^N=\frac{-1\pm \sqrt{5}}{2}$. Firstly, we have
$$\operatorname{Res}_p \left(N \left\{\dfrac{y}{x^2} \right\}_2 \otimes x^N \right)=0 $$
since $\ord_p(x)=0$. Secondly, we have
\begin{align*}
{\rm Res}_p\left(\{x^N\}_2 \otimes \left(1- \frac{y}{x^2} \right) \right)= \ord_p\left(1- \frac{y}{x^2} \right)  \{x_0^N\}_2.  
\end{align*}
Since $\ord_p(1-x^{-2}y) \geq 1$, we deduce $\{x_0^N\}_2 \in \ker(\delta_{\Q(p)})$. But $\{x_0^N\}_2$ may also be viewed in $B_2(\Q(\sqrt{5}))$, and the natural map $\Lambda^2 \Q(\sqrt{5})^\times \to \Lambda^2 \Q(p)^\times$ is injective after tensoring with $\Q$. This implies
$$\{ x_0^N \}_2\in \ker(\delta_{\Q(\sqrt{5})})\simeq K_3^{(2)}(\Q(\sqrt{5})).$$
The latter group is trivial by Borel's theorem \cite[Theorem IV.1.18]{Wei13} 
since $\Q(\sqrt{5})$ is totally real. We conclude that the residue of $\xi_N^{-2,1}$ at $p$ is trivial.
\end{proof}

We now turn our attention to the twist of $\xi_N^{-2,1}$ by the root of unity $-1$.

\begin{cor} 
If $N$ is even, then the elements $\Xi^{-2, 1}_N(-1)$ and $\Xi^{1, -2}_N(-1)$ define elements of $K_4^{(3)}(X_N)$. The elements $\Xi_N^{1, 1}(-\e)$ and $\Xi_N^{-2, 1}(- \e^{-2})$ define elements of $K_4^{(3)}(X_{N,K})$.
\end{cor}

\begin{proof}
Since $N$ is even, by Proposition \ref{dist:cocycle}, we see that
    \[
    \Xi_N^{-2,1}(-1)=\pi^*_{N,N/2}(\Xi^{-2,1}_{N/2})-\Xi_N^{-2,1} \in K_4^{(3)}(X_N)
    \]
    by Theorem \ref{prop:res1}. We conclude using Proposition \ref{iota}.
\end{proof}

\begin{rmk} 
A couple of remarks. 
\begin{itemize}
    \item Going through the proof of Theorem \ref{prop:res1}, we may see that the residues of $\xi_N^{r, s}$ are generally non-trivial, and therefore $\Xi_N^{r,s} \in K_4^{(3)}(U_N^{r,s})$  
    does not extend to $X_N$ (compare with Theorem \ref{rmk:K4X}). For example, let $r=s=1$. Let $p=(x_0,y_0) \in X_N$ be a zero of the function $1-xy$, so that $x_0^N=\frac{1 \pm \sqrt{-3}}{2}$. Then
    $$\operatorname{Res}_p(\xi_N^{1,1}) = \ord_{p}(1-xy) \{x_0^N\}_2. $$
    Note that $x_0^N$ is a primitive $6$-th root of unity, and it is known that $\{\zeta_6\}_2$ is non-zero in the Bloch group of $\Q(\zeta_6)$ \cite[Section 9.B]{Zagier} (in fact, it is a basis since $K_3^{(2)}(\Q(\zeta_6))$ has dimension $1$ by Borel's theorem \cite[Theorem IV.1.18]{Wei13}). Since the canonical map $K_3^{(2)}(\Q(\zeta_6)) \to K_3^{(2)}(\Q(p))$ is injective, \footnote{This follows from the existence of a transfer map $K_3(\Q(p)) \to K_3(\Q(\zeta_6))$. If $f : k \to k'$ is a finite field extension, the projection formula \cite[\S 4, (5)]{Qui} gives $f^* f_* = \deg(f) \cdot \mathrm{id}$ on $K_3(k)$, implying the injectivity of $f_*$ after tensoring with $\Q$.} 
we see that $\xi_N^{1,1}$ has non-trivial residue at $p$.
    \item For $r, s \in \Z$, 
define a $K_2$ element by 
$$e_N^{r, s}(\zeta)=\{1-\zeta x^r y^s, x\} \in K_2(\Q_{\zeta} (X_N)) \otimes \Q. $$   
In general, the tame symbols of $e_N^{r, s}(\zeta)$ do not vanish, \emph{i.e.}~$e_N^{r, s}(\zeta)$ does not extend to an element of $K_2(X_{N, \Q_{\zeta}})\otimes \Q$. 
Theorem \ref{prop:res1} is a $K_4$ analogue of the fact that $e_N^{1, 1}(1)$, $e_N^{-2, 1}(\e)$ and $e_N^{1, -2}(\e)$ extend to elements of $K_2(X_{N, \Q(\e)})\otimes \Q$ (see \cite{Ross91}). 
\end{itemize}
\end{rmk}

\section{Regulator integrals and non-vanishing results}\label{sec:asy}

Let $X$ be a smooth projective geometrically connected curve defined over $\Q$, and let $U=X \setminus S$, where $S$ is a finite set of closed points of $X$. 
We will write $U(\C)=U \times_{\Q} \C$ for the complex points of $U$. Define $H^2(\Gamma(U,3))$ to be the intersection of the kernels of the residue maps $\mathrm{Res}_p$ \eqref{map:res_cohom} for all $p\in U$. In other words, $H^2(\Gamma(U,3))$ fits into an exact sequence
\[
0\longrightarrow H^2(\Gamma(U,3)) \longrightarrow H^2(\Gamma(\Q(X),3)) \overset{\sum \mathrm{Res}_p}{\longrightarrow} \bigoplus_{p\in U} \ker(\delta_{\Q(p)}).
\]
Then we have the Goncharov regulator map \cite[Theorem 2.2]{Goncharov2}
$$r_3(2) \colon H^2(\Gamma(U, 3)) \longrightarrow H^1(U(\C), \R(2))^{+}, $$
where $+$ denotes the invariants with respect to complex conjugation acting on the complex points $U(\C)$ and on the coefficients. 
Explicitly, we have 
\begin{align} \label{regulator} 
r_3(2) (\{f\}_2 \otimes g):= -D_2(f) d\arg(g) -\frac13 \left(-\log|1-f| d\log|f| + \log |f| d\log|1-f| \right) \cdot \log|g|, 
\end{align}
where $D_2 \colon \mathbb{P}^1(\C) \to \R$ is the Bloch--Wigner dilogarithm \eqref{BWlog}. On the other hand, Beilinson \cite{Bei80, Nek94} defined a regulator map 
\begin{equation}\label{BeiReg}
    r_B \colon K_4^{(3)}(U) \longrightarrow H^1(U(\C), \R(2))^{+}.
\end{equation}
The regulator maps of Beilinson and Goncharov are known to be compatible (up to a factor $2$ and a sign) via the map \eqref{map:deJeu} as a consequence of de Jeu's computations of regulator integrals in \cite{Jeu96,Jeu00}. We will only need the following statement.

\begin{thm}[de Jeu] \label{thm:reg_comp}
    In the case $U=X$, we have the equality $\frac 12 r_B \circ \varphi_{\mathrm{dJ}} = r_3(2)$ up to sign.
\end{thm}

\begin{proof}
    This is a consequence of \cite[Theorem 5.4]{Jeu00} since the base field $\Q$ is totally real (see \cite[Section 6.1]{Bru20} for more details). 
\end{proof}

Returning to the case of the Fermat curve $X_N$, we define the homology cycle
\begin{equation}\label{def:gammaN}
\gamma_N := \frac{1}{N^2} \sum_{(r,s) \in G_N} (1-g_N^{r,0})(1-g_N^{0,s}) \delta_N
\end{equation}
where $\delta_N : [0,1] \to X_N$ is the path defined by $$\delta_N(t) := (x(t), y(t)) = (t^{1/N}, (1-t)^{1/N}).$$ 
Then $H_1(X_N(\C), \Q)$ is a cyclic $\Q[G_N]$-module generated by $\gamma_N$ \cite[Proposition 4.9]{Ots11}. We are going to prove non-triviality of $\Xi_N^{1,1}(\varepsilon)\in K_4^{(3)}(X_{N, K})$ by showing that its image under the regulator map $r_B$ is non-trivial. 
Since $r_B$ and $r_3(2)$ are compatible in this case by Theorem \ref{thm:reg_comp}, it suffices to prove the non-triviality of $r_3(2)(\xi^{1,1}_N(\e))$. 
We first give a manageable expression for the regulator integral of $\xi_N^{1,1}(\varepsilon)$ over $\gamma_N$.

\begin{lem} \label{lem:reg_xi1}
Let $\varepsilon$ be a $2N$-th root of unity with $\varepsilon^N = -1$. Then
\begin{equation*}
    \int_{\gamma_N} r_3(2)(\xi_N^{1,1}(\varepsilon)) = \int_{\delta_N} r_3(2)(\xi)
\end{equation*}
with
\begin{equation*}
    \xi = \xi_N^{1,1}(\varepsilon) - \frac{1}{N} \sum_{r \in \Z/N\Z} \xi_N^{1,1}(\varepsilon \zeta_N^r).
\end{equation*}
\end{lem}

\begin{proof}
    By definition of $\gamma_N$ \eqref{def:gammaN}, we have
    \begin{equation*}
        \int_{\gamma_N} r_3(2)(\xi_N^{1,1}(\varepsilon)) = \int_{\delta_N} r_3(2)(\xi)
    \end{equation*}
    with
    \begin{align*}
        \xi = \frac{1}{N^2} \sum_{r,s \in \Z/N\Z} ((1-g_N^{r,0})(1-g_N^{0,s}))^* \xi_N^{1,1}(\varepsilon).
    \end{align*}
    Using Proposition \ref{prop:trace}(1), we compute
    \begin{align*}
        \xi & = \xi_N^{1,1}(\varepsilon) + \frac{1}{N^2} \sum_{r,s \in \Z/N\Z} (-\xi_N^{1,1}(\varepsilon \zeta_N^r) - \xi_N^{1,1}(\varepsilon \zeta_N^s) + \xi_N^{1,1}(\varepsilon \zeta_N^{r+s})) \\
        & = \xi_N^{1,1}(\varepsilon) - \frac{1}{N} \sum_{r \in \Z/N\Z} \xi_N^{1,1}(\varepsilon \zeta_N^r). \qedhere
    \end{align*}
\end{proof}

By Lemma \ref{lem:reg_xi1}, it is enough to express the integral of the regulator of $\xi_N^{1,1}(\zeta)$ along $\delta_N$, where $\zeta$ is a root of unity. This is done in the following lemma.

\begin{lem} \label{lem:reg_xi2}
    Let $\zeta$ be a root of unity. Then
    \begin{align*}
        \int_{\delta_N} r_3(2)(\xi_N^{1,1}(\zeta)) = \int_0^1 \rho_{N,\z}^1 + \rho_{N,\z}^2,
    \end{align*}
    where
    \begin{align*}
        \rho_{N,\z}^1 & = - \frac13 \log \bigl|1-\zeta (t(1-t))^{1/N}\bigr| \Bigl(-\frac{\log(1-t)}{t} - \frac{\log t}{1-t}\Bigr) dt, \\
        \rho_{N,\z}^2 & = \frac13 \log t \Bigl(-\log \bigl|1-\zeta (t(1-t))^{1/N}\bigr| \dlog(t(1-t)) + \log(t(1-t)) \dlog \bigl|1-\zeta (t(1-t))^{1/N}\bigr| \Bigr).
    \end{align*}
\end{lem}

\begin{proof}
    Recall from \eqref{regulator} that
    \begin{equation*}
        r_3(2)(\{f\}_2 \otimes g) = -D_2(f) \darg(g) - \frac13 \log |g| \bigl(-\log |1-f| \dlog |f| + \log |f| \dlog |1-f| \bigr),
    \end{equation*}
    where $D_2$ is the Bloch--Wigner dilogarithm \eqref{BWlog}.
    The function $D_2$ vanishes on $\R$, so that the restriction of $D_2(x^N)$ to $\delta_N$ is zero. Similarly, the restriction of $\darg(x^N)$ to $\delta_N$ is also zero. It follows that in the integral to be computed, the terms involving the dilogarithm disappear. We get
    \begin{equation}\label{eq1}
        \int_{\delta_N} r_3(2)(\xi_N^{1,1}(\zeta)) = \int_{\delta_N} \tilde\rho_{N,\z}^1 + \tilde\rho_{N,\z}^2
    \end{equation}
    with
    \begin{equation}\label{eq2}
    \begin{split}
        \tilde\rho_{N,\z}^1 & = - \frac13 \log \bigl|1-\zeta xy\bigr| \bigl(-\log(1-x^N) \dlog(x^N) + \log(x^N) \dlog(1-x^N)\bigr), \\
        \tilde\rho_{N,\z}^2 & = \frac{N}{3} \log(x^N) \Bigl(-\log \bigl|1-\zeta xy\bigr| \dlog(xy) + \log(xy) \dlog \bigl|1-\zeta xy\bigr| \Bigr).
    \end{split}
    \end{equation}
    The result follows by setting $\rho_{N,\z}^j=\delta_N^* \tilde\rho_{N,\z}^j$ for $j=1,2$, \emph{i.e.}~by setting $x = t^{1/N}$ and $y = (1-t)^{1/N}$, where $t$ goes from $0$ to $1$.
\end{proof}

\begin{thm}\label{regintN2}
    Let $\e$ be a $2N$-th root of unity with $\e^N=-1$. Then 
    \begin{equation*}
    \int_{\gamma_N} r_3(2)(\xi_{N}^{1,1}(\e))=\frac23  \int_0^1 \left(f_{N, \e}(t) - \frac{f_{1,-1}(t)}{N}\right) dt
      + \left(\frac1N \mathcal{L}_3(-1/4)-N^2 \mathcal{L}_3(\e (1/4)^{1/N})\right),
\end{equation*}
where $\mathcal{L}_3$ is Zagier's trilogarithm \eqref{GonchL3} and, for $x\in [0,1[$ and $z\in S^1$, the function $f_{N, z}$ is given by
\begin{equation*}
    f_{N,z}(x):=\Re(\log(1- z (x(1-x))^{1/N})) \frac{\log(1-x)}{x}.
\end{equation*}
\end{thm}

\begin{proof}
    We simplify the notation by writing $\z=\z_N$, $\xi_{N, \e}=\xi_N^{1,1}(\e)$, and $\xi_{N, \e, r}=\xi_N^{1,1}(\e \zeta_N^r)$ in this proof. By Lemma \ref{lem:reg_xi1}, we have 
\[
\int_{\gamma_N} r_3(2)(\xi_{N,\e})=\int_{\delta_N} r_3(2)(\xi_{N, \e}) - \frac{1}{N} \sum_{r \in \Z/N\Z} \int_{\delta_N} r_3(2)(\xi_{N, \e, r}).
\]
By Lemma \ref{lem:reg_xi2}, we have 
\[
\int_{\delta_N} r_3(2)(\xi_{N,\e, r})=\int_0^1 \rho_{N,\e, r}^1 + \int_0^{1} \rho_{N,\e, r}^2=\int_0^1 \rho_{N,\e, r}^1 + \int_0^{1/2} (\rho_{N,\e, r}^2 + \rho_{N,\e, r}^3),
\]
where 
\begin{align*}
    \rho_{N,\e, r}^1 & = -\frac13 \log |1- \e\zeta^r((t(1-t))^{1/N})| \Bigl( -\frac{\log (1-t)}{t} - \frac{\log t}{1-t} \Bigr) dt \\
    \rho_{N,\e, r}^2 & = \frac{1}{3} \log t \bigl( -\log |1 -\e \zeta^r((t(1-t))^{1/N}| \dlog (t(1-t)) + \log (t(1-t)) \dlog |1 -\e \zeta^r((t(1-t))^{1/N}|\bigr) \\
    \rho_{N,\e, r}^3 & = \frac{1}{3} \log (1-t) \bigl( -\log |1 -\e \zeta^r((t(1-t))^{1/N}| \dlog (t(1-t)) + \log (t(1-t)) \dlog |1 -\e \zeta^r((t(1-t))^{1/N}|\bigr).
\end{align*}
We first compute that
\begin{align*}
    \int_0^1 \rho_{N, \e}^1 & = \frac{1}{3}\int_0^1 \log|1-\e(t(1-t))^{1/N}| \left( \frac{\log(1-t)}{t} + \frac{\log(t)}{1-t} \right)dt \\ 
    & = \frac{1}{3}  \int_0^1 \Re(\log(1-\e(t(1-t))^{1/N})) \left( \frac{\log(1-t)}{t} + \frac{\log(t)}{1-t} \right)dt \\ 
    & =\frac{2}{3} \int_0^{1} \Re(\log(1 -\e(t(1-t))^{1/N})) \frac{\log(1-t)}{t} dt  \\
    & = \frac{2}{3}\int_0^{1} f_{N,\e}(t) dt.
\end{align*}
Next, we note that 
\[
\rho_{N, \e}^2+\rho_{N, \e}^3=\frac{1}{3} \log (t(1-t)) \bigl( -\log |1 -\e ((t(1-t))^{1/N}| \dlog (t(1-t)) + \log (t(1-t)) \dlog |1 -\e ((t(1-t))^{1/N}|\bigr).
\]
Putting $v=(t(1-t))^{1/N}$, we see that 
\begin{align*}
\rho_{N, \e}^2+\rho_{N, \e}^3 &=\frac{1}{3} \log (v^N) \bigl( -\log |1 -\e v| \dlog (v^N) + \log (v^N) \dlog |1 -\e v|\bigr) \\
&=\frac{1}{3} \Re\left(\log (v^N) \bigl( -\log (1 -\e v)\dlog (v^N) + \log (v^N) \dlog (1 -\e v)\bigr) \right)
\end{align*}
When $t$ ranges over $]0, 1/2[$, $v$ ranges over $]0,(1/4)^{1/N}[$, and 
we deduce that 
\begin{align*}
\int_0^{1/2} (\rho_{N, \e}^2+\rho_{N, \e}^3) & = \frac{N}{3} \Re \left(\int_{0}^{(1/4)^{1/N}} \log (v) \Bigl( -N\log (1 -\e v) \frac{dv}{v} + N\log (v) \dlog (1 - \e v)\Bigr) \right) \\
& = \frac{N^2}{3} \Re \left(\int_{0}^{(1/4)^{1/N}} \log (v) \Bigl( - \frac{\log (1 -\e v)}{v} - \frac{\e \log (v)}{1 -\e v} \Bigr) dv \right) \\
& = \frac{N^2}{3} \Re \left(\int_{0}^{(1/4)^{1/N}} g_{\e}(v) dv \right),
\end{align*}
where, for $x\in ]0,1]$ and $z\in S^1$, the function $g_{z}$ is defined by
\begin{equation}
    g_{z}(x):=- \log (x) \Bigl(  \frac{z \log (x)}{1-z x} +\frac{\log (1-z x)}{x} \Bigr).
\end{equation}
Define, for $x\in [0, 1]$ and $z\in S^1$, the function 
\begin{equation}\label{fct:Gzx}
    G_{z}(x):=-3\mathrm{Li}_3(z x)+3\mathrm{Li}_2(z x)\log(x)+\log(1-z x)(\log(x))^2,
\end{equation}
where the trilogarithm $\mathrm{Li}_3$ is defined as in \eqref{def:Lik}. One readily checks using \eqref{polylog:deriv} that $G_{z}'(x)=g_z (x)$. Moreover, since $\vert z\vert = 1$, we have
\begin{equation}\label{GtoL3}
\Re(G_z(x))=-3\mathcal{L}_3(z x),
\end{equation}
where $\mathcal{L}_3$ is Zagier's trilogarithm \eqref{GonchL3}.
It follows that 
\begin{equation*}
\int_0^{1/2} (\rho_{N, \e}^2+\rho_{N, \e}^3) = \frac{N^2}{3} \Re \left(  G_{\e}((1/4)^{1/N}) \right)=-N^2 \mathcal{L}_3(\e (1/4)^{1/N}).
\end{equation*}
Finally, we note that 
\begin{equation*}
\sum_{r\in \Z/N\Z} \log \vert 1 -\e \zeta^r (t(1-t))^{1/N} \vert = \log \left\vert \prod_{r\in \Z/N\Z} (1 -\e \zeta^r (t(1-t))^{1/N})  \right\vert  = \log \left\vert 1+t(1-t)  \right\vert.
\end{equation*}
Thus, 
\begin{align*}
\sum_{r\in \Z/N\Z} \int_0^1 \rho_{N,\e, r}^1 &= -\frac13 \int_0^1 \log (1+t(1-t)) \Bigl( -\frac{\log (1-t)}{t} - \frac{\log t}{1-t} \Bigr) dt \\ &= \frac{2}{3} \int_0^1 \log(1+t(1-t))\frac{\log(t)}{1-t} dt \\
&=\frac{2}{3} \int_0^1 f_{1,-1}(t)dt.
\end{align*}
Furthermore, we have 
\[
\sum_{r\in \Z/N\Z} (\rho_{N,\e, r}^2+\rho_{N,\e, r}^3)=\frac{1}{3} \log (t(1-t)) \bigl( -\log (1+ t(1-t)) \dlog (t(1-t)) + \log (t(1-t)) \dlog (1+ t(1-t)).
\]
Putting $u=t(1-t)$, we see that 
\[
\sum_{r\in \Z/N\Z} (\rho_{N,\e, r}^2+\rho_{N,\e, r}^3)=\frac{1}{3} \log (u) \bigl( -\log (1+ u) \dlog (u) + \log (u) \dlog (1+ u)\bigr)=\frac 13 g_{-1}(u).
\]
As $t$ ranges in $]0, 1/2[$, $u$ ranges in $]0, 1/4[$, and we obtain 
\[
\sum_{r\in \Z/N\Z} \int_0^{1/2} (\rho_{N,\e, r}^2+\rho_{N,\e, r}^3)=\frac{1}{3} \int_{0}^{1/4} g_{-1}(u) du=\frac{1}{3} G_{-1}(1/4)=-\mathcal{L}_3(-1/4).
\]
The result follows by putting everything together.
\end{proof}

\begin{thm}\label{thm:asy}
    Let $\e=-1$ if $N$ is odd and $\e=-e^{\frac{\pi i}{N}}$ if $N$ is even. We then have the asymptotic formula
    \[
    \int_{\gamma_N} r_3(2)(\xi_N^{1,1}(\e)) \sim \frac{3}{4} \zeta(3) N^2, \qquad \text{ as } N\to +\infty.
    \]
\end{thm}

\begin{proof}
    By Theorem \ref{regintN2}, it suffices to show that 
    \[
    \frac{2}{3N^2} \int_0^1 f_{N,\e}(t)dt = o(N^2), \qquad \text{ as } N\to +\infty,
    \]
    since $\e (1/4)^{1/N} \to -1$ as $N\to +\infty$ and $-\mathcal{L}_3(-1)=-\mathrm{Li}_3(-1)=(1-2^{-2})\zeta(3)=\frac{3}{4} \zeta(3)$.

    We have
    \[
f_{N,\e}(x)=
\begin{cases}
    \log(1+(x(1-x))^{1/N}) \frac{\log(1-x)}{x}, & N\equiv 1\mod 2; \\
    \frac12 \log (1+2 \cos (\frac{\pi}N) (x(1-x))^{\frac1N} + (x(1-x))^{\frac2N}) \frac{\log(1-x)}x, & N\equiv 0\mod 2.
\end{cases}
\]
    
Note that $f_{N, \e}(x)<0$ for all $x\in ]0,1[$. Moreover,
if $N_1<N_2$, then $f_{N_1, \e}(x)> f_{N_2, \e}(x)$ for all $x\in ]0,1[$. Finally, since $((t(1-t))^{1/N}\to 1$ as $N\to +\infty$, we have 
\[
\lim_{N\to + \infty} f_{N, \e} = f,
\]
where, for $x\in ]0,1[$, the function $f$ is defined by
\[
f(x):=\log(2)\frac{\log(1-x)}{x}.
\]
As a consequence, the sequence $\{\frac23 \int_0^1 f_{N,\e}(t)dt\}_{N\geq 1}$ is monotonously decreasing with limit given by 
\begin{multline*}
    \lim_{N\to +\infty} \frac23  \int_0^1 f_{N,\e}(t)dt = \frac{2}{3}\int_0^{1} f(t) dt=\frac{2\log(2)}{3}\int_0^1 \frac{\log(1-t)}{t} dt \\ = -\frac{2\log(2)}{3} \mathrm{Li}_2(1) =-\frac{2\log(2)}{3}\zeta(2)=-\frac{\log(2)\pi^2}{9}.
\end{multline*}
\end{proof}

\begin{cor}
    For large enough $N$, the element $\Xi^{-2,1}_N\in K_4^{(3)}(X_{N})$ is non-trivial.
\end{cor}

\begin{proof}
    Direct consequence of Proposition \ref{cor:pb}(iii) and Theorem \ref{thm:asy} upon invoking de Jeu's compatibility result for regulator maps (Theorem \ref{thm:reg_comp}).
\end{proof}

By being a bit more precise in our analysis of integrals, we can prove a much stronger result.

\begin{thm} \label{thm:nontrivial}
    For all $N\geq 3$, the element $\Xi_N^{-2,1} \in K_4^{(3)}(X_N)$ is non-trivial.
\end{thm}

\begin{proof}
    By Theorem \ref{thm:reg_comp}, it suffices to prove that $r_3(2)(\xi^{-2,1}_N)$ is non-trivial. By Proposition \ref{cor:pb}, it in turn suffices to prove that $r_3(2)(\xi^{1,1}_N(\e))$ is non-trivial for some $2N$-th root of unity $\e$ such that $\e^N=-1$. Choose $\e=-1$ if $N$ is odd and $\e=-e^{\frac{\pi i}{N}}$ if $N$ is even.

    Suppose first that $N$ is odd. By Theorem \ref{regintN2}, we have 
    \begin{equation*}
    \int_\gamma r_3(2)(\xi_N^{1,1}(-1))=\frac{2}{3}\int_0^1 \Bigl(f_{N,-1}(t) - \frac{f_{1,-1}(t)}{N}\Bigr)dt+\frac{1}{3}\Bigl( N^2 G_{-1}((1/4)^{1/N}) -\frac{1}{N} G_{-1}(1/4) \Bigr),
\end{equation*}
where $G_{-1}$ is the function defined by \eqref{fct:Gzx} and which is related to the function $\mathcal{L}_3$ via \eqref{GtoL3}.
We established during the course of the proof of Theorem \ref{thm:asy} that the sequence $\{\frac23 \int_0^1 f_{N,-1}(t)dt\}_{N\geq 1}$ is monotonously decreasing with limit given by $-\frac{\log(2)\pi^2}{9}$. Thus, for all $N\geq 3$, we have
\begin{align*}
    \int_\gamma r_3(2)(\xi_N^{1,1}(-1)) & > -\frac{\log(2)\pi^2}{9} -\frac{2}{3N}\int_0^1 f_1(t) dt +\frac{1}{3}\Bigl( N^2 G_{-1}((1/4)^{1/N}) -\frac{1}{N} G_{-1}(1/4) \Bigr) \\
    & > -\frac{\log(2)\pi^2}{9} +\frac{1}{3}\Bigl( N^2 G_{-1}((1/4)^{1/N}) -\frac{1}{N} G_{-1}(1/4) \Bigr) \\
    & > -\frac{\log(2)\pi^2}{9} +\frac{1}{3}\Bigl( N^2 G_{-1}(1/2) -\frac{1}{N} G_{-1}(1/4) \Bigr) \\
    & > -\frac{\log(2)\pi^2}{9} +\frac{1}{3}\Bigl( 4 G_{-1}(1/2) -\frac{1}{2} G_{-1}(1/4) \Bigr) \\
    & > -\frac{\log(2)\pi^2}{9} +3 \\
    & > -0.77 +3 \\
    & > 0.
\end{align*}
In the above inequalities, we chronologically used the fact that $\int_0^1 f_{N,-1}(t)dt > -\log(2)\pi^2/9,$ the fact that $-f_{1,-1}(x)>0$ for all $x\in ]0,1[$, the fact that $G_{-1}'(x)=g_{-1}(x)>0$ for all $x\in ]0,1[$ and thus $G_{-1}(x)$ is strictly increasing on this interval, coupled with the fact that $(1/4)^{1/N}>1/2$ since $N>2$, and finally the fact that $N>2$, followed by the estimates 
\[
\frac{1}{3}\Bigl( 4 G_{-1}(1/2) -\frac{1}{2} G_{-1}(1/4) \Bigr) \simeq 3.03\ldots \qquad \text{ and } \qquad -\frac{\log(2)\pi^2}{9} \simeq -0.76\ldots
\]
The non-triviality of $\int_\gamma r_3(2)(\xi_N^{1,1}(-1))$ implies the desired non-triviality of $r_3(2)(\xi_N^{1,1}(-1))$.

    Next, we turn to the case where $N$ is even. 
    Assume, by contradiction, that $\xi_N^{-2,1}=0$. Then, by Proposition \ref{dist:cocycle} and Proposition \ref{prop:trace}(i), we have 
    \[
    \pi_{N,N/2}^* (\xi_{N/2}^{-2,1}) = \xi_N^{-2,1}+\xi_N^{-2,1}(-1) = \xi_N^{-2,1}+\xi_N^{-2,1}(\zeta_N^{N/2})= (1+g^{0,N/2}) (\xi_N^{-2,1}) = 0.
    \]
    The map $\pi^*_{N, N/2}$ is injective since 
    $(\pi_{N, N/2})_* \circ \pi^*_{N, N/2}= \deg{\pi_{N, N/2}= 2},$ so the above calculation implies that $\xi_{N/2}^{-2,1}=0$. Repeating the process, we eventually reach the conclusion that $\xi_{M}^{-2,1}=0$ for some odd $M\geq 3$ or $\xi_{4}^{-2,1}=0$. The first option contradicts the first part of the proof. We deduce that $\xi_{4}^{-2,1}=0$. In Section \ref{sec:num} below, we numerically compute the values
    \[
    R_N = \int_{\gamma_N} r_3(2)(\xi_N^{-2,1})
    \]
    for various $N$. The value $R_4$ is recorded in Table \ref{table} and is non-zero. We have arrived at a contradiction and may thus conclude that $\xi_N^{-2,1}\neq 0$.
\end{proof}

\section{Regulator formulas via hypergeometric functions}\label{s:hyperg}

Let us normalize the differential form $\omega_N^{a, b}$ \eqref{omegars} in the following way:
\begin{equation*}
\widetilde{\omega}_N^{a, b}:= \left(\frac1N B\left(\frac{\langle a \rangle}N, \frac{\langle b \rangle}N \right) \right)^{-1} \omega_N^{a, b},
\end{equation*}
where the Beta function is the function defined by
\[
B(z_1,z_2):=\frac{\Gamma(z_1)\Gamma(z_2)}{\Gamma(z_1+z_2)}.
\]
By \cite[Proposition 4.9]{Ots11}, we then have
\begin{equation}\label{eq:normint}
     \int_{\gamma_N} \tilde\omega_N^{a,b}=1.
\end{equation}
Thus, for any $g_N^{r, s} \in G_N$, we have
\begin{align*}
\int_{g_N^{r, s} \gamma_N} \widetilde{\omega}_N^{a, b} = \int_{\gamma_N} g_N^{r, s} \widetilde{\omega}_N^{a, b}  = \zeta_N^{ar+bs}. 
\end{align*}
Hence, denoting by $c_{\infty}$ the complex conjugation on the coefficients, we have 
$$c_{\infty} \widetilde{\omega}_N^{a, b} = \widetilde{\omega}_N^{-a, -b}. $$

\begin{dfn}
Define a function of a positive real number $\al$ by 
$$F(\al;  x)=  F_1(\al; x)-F_2(\al; x), $$
where 
\begin{align*}
&F_1(\al; x):=\dfrac{\Gamma\left(\al \right)^2}{\Gamma\left(2 \al \right)}\sum_{i, n \ge 0} \left(\dfrac{2}{n+1} - \dfrac{1}{\al+i} \right) \dfrac{(\al, i)(\al, n+i+1)}{(2\al, n+2i+2)} x^i, \\
&F_2(\al; x):=\dfrac{\Gamma\left(\al \right)^2}{\Gamma\left(2\al \right)}\sum_{i, m, n \ge 0} \left(\dfrac1{m+1} + \dfrac1{n+1} \right) \dfrac{(\al, i)(\al, m+n+i+2) }{(2\al, m+n+2i+3)} x^i, 
\end{align*}
and $(\al, i) = \Gamma(\al+i)/\Gamma(\al)$ denotes the rising Pochhammer symbol. 
\end{dfn}

\begin{prop} \label{convergence}
The function $F(\alpha; x)$ converges absolutely for $|x|=1$. 
\end{prop}

To prove the proposition, we prepare some hypergeometric functions. 
Let $${_{p+1}F_{p}}\left[ \left. 
\begin{matrix}
\alpha_1, \alpha_2, \ldots, \alpha_{p+1} \\
\beta_1, \beta_2, \ldots, \beta_p
\end{matrix}
\right| x
\right]
:=
\sum_{n \ge 0} \dfrac{\prod_{i=1}^{p+1}(\alpha_i, n)}{\prod_{i=1}^p(\beta_i, n)} \dfrac{x^n}{n!}$$
be a generalized hypergeometric function, which converges absolutely for $|x|<1$, and converges at $x=1$ if 
\begin{align} \label{convergence_condition1}
\operatorname{Re} \left(\sum_{j=1}^{p} \beta_j -\sum_{i=1}^{p+1} \alpha_i\right) > 0
\end{align}
(see \cite{Slater66}). 
Let 
\begin{multline*}
F^{A;B;B'}_{C;D;D'}\left[\left.
\begin{matrix}
a_1, \ldots, a'_A \\
c_1, \ldots, c_C 
\end{matrix}
; 
\begin{matrix}
b_1, \ldots, b_B \\
d_1, \ldots, d_D
\end{matrix}
;
\begin{matrix}
B'_1, \ldots, B'_{B'} \\
d'_1, \ldots d'_{D'}
\end{matrix}
\right| x, y
\right] \\
:=
\sum_{m, n \ge 0}\dfrac{\prod_{i=1}^A(a_i, m+n) \prod_{i=1}^B(b_i, m) \prod_{i=1}^{B'}(b'_i, n)}{\prod_{i=1}^C(c_i, m+n) \prod_{i=1}^{D} (d, m)\prod_{i=1}^{D'}(d, n)} \dfrac{x^my^n}{m!n!}  
\end{multline*}
be a Kamp\'e de F\'eriet hypergeometric function. 
In this paper, we consider the case $A=C$, $B=D+1$ and $B'=D'+1$. 
The double series $F^{A;B+1, C+1}_{A; B; C}(x, y)$ converges absolutely for $|x|<1$ and $|y|<1$, and 
converges at $|x| = |y|=1$ if 
\begin{equation}\label{convergence_condition2}
    \begin{cases}
        \operatorname{Re} \left(\sum_{i=1}^A c_i + \sum_{i=1}^B d_i - \sum_{i=1}^A a_i -\sum_{i=1}^{B+1} b_i \right) >0, & \\
\operatorname{Re} \left(\sum_{i=1}^A c_i + \sum_{i=1}^C d'_i - \sum_{i=1}^A a_i -\sum_{i=1}^{B+1} b'_i \right) >0, & \\
\operatorname{Re} \left(\sum_{i=1}^A c_i +\sum_{i=1}^B d_i + \sum_{i=1}^C d'_i - \sum_{i=1}^A a_i -\sum_{i=1}^{B+1} b_i  -\sum_{i=1}^{B+1} b'_i \right)>0 &
    \end{cases}
\end{equation} 
(see \cite[Theorem 1]{HMS92}). 

\begin{proof}[Proof of Proposition \ref{convergence}]
First, we prove that $F_1(\alpha; x)$ converges absolutely for $|x|=1$.  
We can write
\begin{align*}
&\sum_{i, n \ge 0} \left(\dfrac{2}{n+1} - \dfrac{1}{\al+i} \right) \dfrac{(\al, i)(\al, n+i+1)}{(2\al, n+2i+2)} x^i \\
&=\dfrac1{2 \alpha+1} \sum_{i, n \ge 0} \dfrac{(\alpha+1, n+i)(\alpha, i)(1, n)}{(2\alpha+2, n + 2i)(2, n)}x^i
-
\dfrac{1}{2 \alpha(2 \alpha+1)} \sum_{i, n \ge 0} \dfrac{(\alpha+1, n+i) (\alpha, i)^2}{(2 \alpha+2, n+2i)(\alpha+1, i)}x^i. 
\end{align*}
We only prove that the first term converges for $|x|=1$ since the second term is proved similarly. 
We can write  
$$\sum_{i, n \ge 0} \dfrac{(\alpha+1, n+i)(\alpha, i)(1, n)}{(2\alpha+2, n + 2i)(2, n)}x^i=
\left.
\sum_{i \ge 0}{_{3}F_{2}}\left[ \left. 
\begin{matrix}
\alpha+1+i,1,1 \\
2\alpha+2+2i, 2
\end{matrix}
\right| y
\right] \dfrac{(\alpha, i)(\alpha+1, i)}{(2\alpha+2, 2i)}x^i \right|_{y=1}.  $$
By \eqref{convergence_condition1}, ${}_3F_2(y):={_{3}F_{2}}\left[ \left. 
\begin{matrix}
\alpha+1+i,1,1 \\
2\alpha+2+2i, 2
\end{matrix}
\right| y
\right]$ converges for $|y|=1$ for all $0 \leq i < \infty$ and
$$\lim_{i \to \infty}{_{3}F_{2}}\left[ \left. 
\begin{matrix}
\alpha+1+i,1,1 \\
2\alpha+2+2i, 2
\end{matrix}
\right| 1
\right] 
={_{2}F_{1}}\left[ \left. 
\begin{matrix}
1,1 \\
2
\end{matrix}
\right| \frac12
\right] < \infty, $$
hence 
${_{3}F_{2}}(1)$
is bounded independently of $i$. 
Therefore, it suffices to show that
$$\sum_{i \ge 0} \dfrac{(\alpha, i)(\alpha+1, i)}{(2\alpha+2, 2i)}x^i$$
converges at $|x|=1$.  
By the duplication formula, we have 
$$(2 \alpha +2, 2i)=2^{2i}(\alpha+1, i) \left(\alpha+\frac32, i \right), $$
hence we obtain 
$$\sum_{i \ge 0} \dfrac{(\alpha, i)(\alpha+1, i)}{(2\alpha+2, 2i)}x^i={_{2}F_{1}}\left[ \left. 
\begin{matrix}
\alpha, 1 \\
\alpha+ \frac32
\end{matrix}
\right| \frac{x}4
\right], $$
which converges at $|x|=1$. 
Hence, the first asserion follows. 

Secondly, we prove that $F_2(\alpha; x)$ converges absolutely for $|x|=1$.   
We can write 
$$F_2(\al; x)=\dfrac{\Gamma\left(\al \right)\Gamma\left(\al\right)}{\Gamma\left(2\al \right)} \cdot \dfrac{1}{2(2 \alpha+1)}
\sum_{i \geq 0}
F_{2;1;1}^{2;2;2}\left[\left.
\begin{matrix}
\alpha+2+i, 3 \\
2 \alpha+3+2i, 2
\end{matrix}
; 
\begin{matrix}
1, 1 \\
2
\end{matrix}
;
\begin{matrix}
1, 1 \\
2
\end{matrix}
\right| y, z
\right]
\dfrac{(\alpha+2, i)(\alpha, i)}{(2\alpha+3, 2i)}x^i
\big|_{(y, z)=(1,1)}. 
$$
By \eqref{convergence_condition2}, $F_{2,1,1}^{2,2,2}(y, z)$ converges absolutely for $|y|=|z|=1$ and $0 \le i < \infty$.  
Since 
$$\lim_{i \to \infty}F_{2,1,1}^{2,2,2}(1, 1)=F_{1;1;1}^{1;2;2}\left[\left.
\begin{matrix}
3 \\
2
\end{matrix}
; 
\begin{matrix}
1, 1 \\
2
\end{matrix}
;
\begin{matrix}
1, 1 \\
2
\end{matrix}
\right| \frac12, \frac12
\right] < \infty, $$
$F_{2,1,1}^{2,2,2}(1, 1)$ is bounded independently of $i$. 
Therefore, it suffices to show that 
$$\sum_{i \ge 0} \dfrac{(\alpha+2, i)(\alpha, i)}{(2 \alpha+3, 2i)}x^i$$
converges absolutely at $|x|=1$. 
By the duplication formula, we have 
$$(2 \alpha+3, 2i)=2^{2i} \left(\dfrac{2 \alpha+3}2, i \right)\left(\alpha+2, i \right), $$
hence 
$$\sum_{i \ge 0} \dfrac{(\alpha+2, i)(\alpha, i)}{(2 \alpha+3, 2i)}x^i={_{3}F_{2}}\left[ \left. 
\begin{matrix}
\alpha, \alpha+2, 1 \\
\alpha+ \frac32, \alpha+1
\end{matrix}
\right| \frac{x}4
\right], $$
which converges absolutely at $|x|=1$. 
Hence, the second assertion follows. 
\end{proof}

The following result can be viewed as a generalization, from the setting of $K_2$ groups to $K_4$ groups of Fermat curves, of the main result of Otsubo \cite[Theorem 4.14]{Ots11}.

\begin{thm}\label{thm:reg_hyperg}
Let $\e$ be a $2N$-th root of unity with $\e^N=-1$. Then 
\[
r_3(2) (\xi_N^{1,1}(\e))=\dfrac1{6N} \sum_{1 \leq a <N} \e^{-a} \left(F\left(\dfrac{a}{N};-1 \right) - F\left(1-\dfrac{a}{N};-1 \right)\right) \widetilde{\omega}_N^{a, a}. 
\]
\end{thm}

\begin{proof}
Combining Lemma \ref{lem:reg_xi1} with equation \eqref{eq1} yields
\[
\int_{\gamma_N} r_3(2)(\xi_N^{1,1}(\varepsilon))=\int_{\gamma_N} (\tilde\rho_{N,\varepsilon}^1+\tilde\rho_{N,\varepsilon}^2),
\]
where, by \eqref{eq2}, we have
\begin{align*}
\tilde\rho_{N,\varepsilon}^1+\tilde\rho_{N,\varepsilon}^2 & =  -\frac13 \left(-\log|1-x^N| d\log|x^N| + \log |x^N| d\log|1-x^N| \right) \cdot \log|1- \e xy| \notag \\
& \qquad + \frac{N}3 \left(-\log|1- \e xy| d\log|xy| + \log |xy| d\log|1- \e xy| \right) \cdot \log|x^N|  \notag \\
&=\Re\left( -\frac13 \left(-\log(1-x^N) d\log(x^N) + \log (x^N) d\log(1-x^N) \right) \cdot \log(1- \e xy)  \right.  \\
& \qquad\left.  + \frac{N}3 \left(-\log(1- \e xy) d\log(xy) + \frac1N \log (xy)^N d\log(1- \e xy) \right) \cdot \log(x^N) \right).  \notag 
\end{align*}
First, we compute the term $\log(1-x^N) d\log x^N \log(1- \e xy)$. 
For $|x|<1$, $|y|<1$, we have 
\begin{multline*}
\log(1-x^N) d\log x^N  \log(1- \e xy)
= N \sum_{m, n \ge 1} \dfrac{(\e xy)^m}{m} \dfrac{(x^N)^n}{n} \dfrac{dx}x \\
=N \sum_{m, n \ge 1} \dfrac{\e^m}{mn} \omega_N^{m+Nn, m+N} \equiv N \sum_{m, n \ge 1} \dfrac{\e^m}{n(2m+Nn)} \omega_N^{m+Nn, m}
\end{multline*}
modulo exact forms by \cite[Lemma 4.19]{Ots11}. 
Put $m=a+Ni$. Then we obtain 
\begin{align*}
&N \sum_{m, n \ge 1} \dfrac{\e^m}{n(2m+Nn)} \omega_N^{m+Nn, m} \\
&= N \sum_{1 \leq a \leq N} \e^a\sum_{i, n \ge 0} \dfrac{(-1)^i}{(n+1)} \dfrac{1}{2a+N(n+2i+1)} \omega_N^{a+N(n+i+1), a+Ni} \\
&\equiv  N \sum_{1 \leq a \leq N} \e^a\sum_{i, n \ge 0} \dfrac{(-1)^i}{(n+1)} \dfrac{\left(\frac{a}N, n+i+1 \right)\left(\frac{a}N, i \right)}{(2a+N(n+2i+1)) \left(\frac{2a}N, n+2i+1 \right)} \omega_N^{a, a} \\
&=  \sum_{1 \leq a \leq N} \e^a\sum_{i, n \ge 0} \dfrac{(-1)^i}{(n+1)} \dfrac{\left(\frac{a}N, n+i+1 \right)\left(\frac{a}N, i \right)}{\left(\frac{2a}N, n+2i+2 \right)} \omega_N^{a, a}. 
\end{align*}
We compute similarly that 
\begin{align*}
\log x^N  d\log (1-x^N) \log(1- \e xy)&= \log (1-y^N)  d\log (y^N) \log(1- \e xy) \\
&=- \sum_{1 \leq a \leq N} \e^a\sum_{i, n \ge 0} \dfrac{(-1)^i}{(n+1)} \dfrac{\left(\frac{a}N, n+i+1 \right)\left(\frac{a}N, i \right)}{\left(\frac{2a}N, n+2i+2 \right)} \omega_N^{a, a}. 
\end{align*}

Next, we compute the term $\log(1- \e xy) d\log (xy) \log(x^N)$. 
We have 
\begin{align*}
\log(1- \e xy) d\log (xy) \log(x^N) 
&=\sum_{m, n \ge 1} \dfrac{(\e xy)^m}{m} \dfrac{(y^N)^n}{n} \left(\dfrac{dx}x + \dfrac{dy}y \right) \\
&=\sum_{m, n \ge 1} \dfrac{(\e xy)^m}{m} \dfrac{(y^N)^n}{n} \dfrac{y^N-x^N}{xy^N}dx \\
&=\sum_{m, n \ge 1} \dfrac{\e^m}{mn} (\omega_N^{m, m+N+Nn} -\omega_N^{m+N, m+Nn}) \\
& \equiv \sum_{m, n \ge 1} \dfrac{N\e^m}{m (2m+Nn)} \omega_N^{m, m+Nn} 
\end{align*}
modulo exact forms.
Put $m=a+Ni$. Then we have 
\begin{align*}
\sum_{m, n \ge 1} \dfrac{N\e^m}{m (2m+Nn)} & \omega_N^{m, m+Nn} =\sum_{1 \leq a \leq N} \e^a \sum_{i, n \ge 0} \dfrac{N(-1)^i}{(a+Ni) (2a+N(n+2i+1)} \omega_N^{a+Ni, a+N(n+i+1)} \\
& \equiv \sum_{1 \leq a \leq N} \e^a \sum_{i, n \ge 0} \dfrac{(-1)^i}{(a+Ni) \left(\frac{2a}{N}+n+2i+1\right) } \dfrac{\left(\frac{a}N, i \right)\left(\frac{a}N, n+i+1 \right)}{\left(\frac{2a}N, n+2i+1 \right)}
\omega_N^{a, a} \\
&= \sum_{1 \leq a \leq N} \e^a \sum_{i, n \ge 0} \dfrac{(-1)^i}{(a+Ni)} \dfrac{\left(\frac{a}N, i \right)\left(\frac{a}N, n+i+1 \right)}{\left(\frac{2a}N, n+2i+2 \right)}
\omega_N^{a, a}. 
\end{align*}

Finally, we deal with the term $ \log (xy) d\log(1- \e xy) \log(x^N)$. 
We compute that 
\begin{align*}
\frac1N & \log  (xy)^N  d\log(1- \e xy) \log(x^N) \\
&=\frac1N\log (x^Ny^N) d\log(1- \e xy) \log(x^N) \\
&=\frac1N\log (1-x^N)(1-y^N) d\log(1- \e xy) \log(x^N) \\
&=\frac1N \sum_{m, n \ge 1} \left(\dfrac{x^{Nm}}{m} + \frac{y^{Nm}}{m} \right) \dfrac{y^{Nn}}{n} \sum_{l \ge 1} (\e xy)^l (-1) \left(\dfrac{y^N-x^N}{xy^N} \right)dx \\
&=- \frac1N \sum_{l, m, n \ge 1}\dfrac{\e^l}{mn}\left(\omega_N^{l+Nm, l+Nn+N} -\omega_N^{l+Nm+N, l+Nn} +\omega_N^{l, l+Nn+Nm+N} -\omega_N^{l+N, l+Nm+Nn} \right) \\
&=-\frac1N \sum_{l, m, n \ge 1}\dfrac{\e^l}{mn} \dfrac{1}{(2l+N(m+n))} \left(N(n-m)  \omega_N^{l+Nm, l+Nn} +N(n+m) \omega_N^{l, l+N(n+m)}\right). 
\end{align*}
Put $l = a+Ni$. Then we have 
\begin{align*}
&-\frac1N \sum_{1 \leq a \leq N} \e^a \sum_{i, m, n \ge 0}\dfrac{(-1)^i}{(m+1)(n+1)} \dfrac{N}{(2a+N(m+n+2i+2))\left(\frac{2a}N, m+n+2i+2 \right)} \\
& \times \left((n-m) \left(\frac{a}N, m+i+1 \right)\left(\frac{a}N, n+i+1 \right)+(n+m+2)\left(\frac{a}N, i \right)\left(\frac{a}N, n+m+i+2 \right) \right) \omega_N^{a, a} \\
&=-\frac1N \sum_{1 \leq a \leq N} \e^a  \sum_{i, m, n \ge 0}\dfrac{(-1)^i}{(m+1)(n+1)} \dfrac{1}{\left(\frac{2a}N, m+n+2i+3 \right)} \\
& \times \left((n-m) \left(\frac{a}N, m+i+1 \right)\left(\frac{a}N, n+i+1 \right)+(n+m+2)\left(\frac{a}N, i \right)\left(\frac{a}N, n+m+i+2 \right) \right) \omega_N^{a, a}. 
\end{align*}
The term 
$$\sum_{i, m, n \ge 0}\dfrac{(-1)^i}{(m+1)(n+1)} \dfrac{(n-m) \left(\frac{a}N, m+i+1 \right)\left(\frac{a}N, n+i+1 \right)}{\left(\frac{2a}N, m+n+2i+3 \right)}$$
is equal to zero since $$\frac{(-1)^i}{(m+1)(n+1)} \frac{(n-m) \left(\frac{a}N, m+i+1 \right)\left(\frac{a}N, n+i+1 \right)}{\left(\frac{2a}N, m+n+2i+3 \right)}$$ is anti-symmetric with respect to $m$ and $n$. 
Putting everything together shows that $\int_{\gamma_N} r_3(2)(\xi_N^{1,1}(\e))$ is the integral over $\gamma_N$ of the differential form
\begin{align*}
 & \Re \Bigl( \frac{2}{3}\sum_{1 \leq a \leq N} \e^a\sum_{i, n \ge 0} \dfrac{(-1)^i}{(n+1)} \dfrac{\left(\frac{a}N, n+i+1 \right)\left(\frac{a}N, i \right)}{\left(\frac{2a}N, n+2i+2 \right)} \omega_N^{a, a} \\
& \qquad -\frac{N}{3} \sum_{1 \leq a \leq N} \e^a \sum_{i, n \ge 0} \dfrac{(-1)^i}{(a+Ni)} \dfrac{\left(\frac{a}N, i \right)\left(\frac{a}N, n+i+1 \right)}{\left(\frac{2a}N, n+2i+2 \right)}
\omega_N^{a, a} \\
& \qquad\qquad -\frac{1}{3}  \sum_{1 \leq a \leq N} \e^a  \sum_{i, m, n \ge 0}\dfrac{(-1)^i}{(m+1)(n+1)} \dfrac{1}{\left(\frac{2a}N, m+n+2i+3 \right)} \\
& \qquad\qquad\qquad\qquad\qquad\qquad \times \left((n+m+2)\left(\frac{a}N, i \right)\left(\frac{a}N, n+m+i+2 \right) \right) \omega_N^{a, a} \Bigr) \\
& \quad = \frac{1}{3} \Re \Bigl(  \sum_{1 \leq a \leq N} \e^a \left( \sum_{i, n \ge 0} (-1)^i \dfrac{\left(\frac{a}N, n+i+1 \right)\left(\frac{a}N, i \right)}{\left(\frac{2a}N, n+2i+2 \right)} \left( \dfrac{2}{n+1} - \dfrac{1}{\frac{a}{N}+i}\right) \right. \\
& \qquad\qquad \left. - \sum_{i, m, n \ge 0}\dfrac{(-1)^i}{(m+1)(n+1)} \dfrac{(n+m+2)\left(\frac{a}N, i \right)\left(\frac{a}N, n+m+i+2 \right)}{\left(\frac{2a}N, m+n+2i+3 \right)} \right)\omega_N^{a, a} \Bigr).
\end{align*}
Observe that 
\begin{multline*}
\e^{a} F\left( \frac{a}{N} ; -1 \right) =\e^a \dfrac{\Gamma\left(\frac{a}{N} \right)^2}{\Gamma\left(\frac{2a}{N} \right)} \left( \sum_{i, n \ge 0} \left(\dfrac{2}{n+1} - \dfrac{1}{\frac{a}{N}+i} \right) \dfrac{(\frac{a}{N}, i)(\frac{a}{N}, n+i+1)}{(\frac{2a}{N}, n+2i+2)} (-1)^i \right. \\
 \left. -\sum_{i, m, n \ge 0} \left(\dfrac1{m+1}+ \dfrac1{n+1} \right) \dfrac{(\frac{a}{N}, i)(\frac{a}{N}, m+n+i+2) }{(\frac{2a}{N}, m+n+2i+3)} (-1)^i  \right).
\end{multline*}
We thus conclude that
\[
\int_{\gamma_N}r_3(2)(\xi_N^{1,1}(\e)) =\frac{1}{3} \Re  \left(\sum_{1 \leq a \leq N} \dfrac{\Gamma\left(\frac{2a}{N} \right)}{\Gamma\left(\frac{a}{N} \right)^2} \e^a F\left( \frac{a}{N} ; -1 \right) \int_{\gamma_N} \omega_N^{a, a} \right).
\]
Using the formula $$\dfrac{\Gamma\left(\frac{2a}{N} \right)}{\Gamma\left(\frac{a}{N} \right)^2}=B\left( \frac{a}{N}, \frac{a}{N}\right)^{-1},$$ we arrive at the formula
\[
\int_{\gamma_N} r_3(2)(\xi_N^{1,1}(\e)) =\frac{1}{3} \Re  \left(\sum_{1 \leq a \leq N} \e^a F\left( \frac{a}{N} ; -1 \right) B\left( \frac{a}{N}, \frac{a}{N}\right)^{-1} \int_{\gamma_N} \omega_N^{a, a} \right),
\]
which in turn yields 
\begin{equation}\label{eq:gammaN_reg_hyper}
\int_{\gamma_N} r_3(2)(\xi_N^{1,1}(\e)) =\frac{1}{3N} \Re  \left(\sum_{1 \leq a \leq N} \e^a F\left( \frac{a}{N} ; -1 \right) \int_{\gamma_N}\tilde{\omega}_N^{a, a}, \right).
\end{equation}
Since $\gamma_N$ generates $H_1(X_N(\C),\Q)$ as a cyclic $\Q[G_N]$-module (see the discussion preceding Lemma \ref{lem:reg_xi1}), we deduce that for all $\gamma\in H_1(X_N(\C),\Z)$ we have 
$$\int_{\gamma} r_3(2)(\xi_N^{1,1}(\e))=\frac1{3N} \Re \left( \sum_{1 \leq a \leq N} \e^a F\left(\frac{a}N; -1\right)\int_{\gamma} \widetilde{\omega}_N^{a, a} \right). $$
Indeed, as $(r,s)$ ranges over $I_N$, the cycles $\kappa^{r, s}_N:=(1-g_N^{r, 0}) (1-g_N^{0, s}) \delta_N$ give a basis of $H_1(X_N(\C), \Q)$ (see \cite[Theorem 4.1]{Ots}). 
On $\kappa_N^{r, s}$, we have $|x|<1$ and $|y|<1$ except at the cusps, hence equation \eqref{eq:gammaN_reg_hyper} holds for any $\gamma$ by an argument similar to \cite{Ots11}. 
Since 
\begin{align*}
\Re \left(\e^a F\left(\frac{a}N;-1 \right)\right) \cdot \Re\left(\widetilde{\omega}_N^{a, a} \right)
= \dfrac12(\e^a+\e^{-a}) F\left( \frac{a}N;-1 \right) \cdot \dfrac12 (\widetilde{\omega}_N^{a, a}+\widetilde{\omega}_N^{N-a, N-a})
\end{align*}
and 
\begin{align*}
\Im \left(\e^a F\left(\frac{a}N;-1 \right)\right) \cdot \Im\left(\widetilde{\omega}_N^{a, a} \right)
= \dfrac1{2i}(\e^a-\e^{-a}) F\left( \frac{a}N;-1 \right) \cdot \dfrac1{2i} (\widetilde{\omega}_N^{a, a}-\widetilde{\omega}_N^{N-a, N-a}), 
\end{align*}
we obtain
\begin{align*}
&\frac1{3N} \Re \left( \sum_{1 \leq a \leq N} \e^a F\left(\frac{a}N; -1\right)\int_{\gamma} \widetilde{\omega}_N^{a, a} \right) \\
& \qquad = 
\dfrac1{12N}
\sum_{1 \leq a \leq N} 
(\e^a + \e^{-a}) \left(F\left(\frac{a}N; -1\right)-F\left(1-\frac{a}N; -1\right) \right) \int_{\gamma} \widetilde{\omega}_N^{a,a} \\
&\qquad\qquad-
\dfrac1{12N}
\sum_{1 \leq a \leq N} 
(\e^a - \e^{-a}) \left(F\left(\frac{a}N; -1\right)-F\left(1-\frac{a}N; -1\right) \right)
\int_{\gamma} \widetilde{\omega}_N^{a,a}. 
\end{align*}
Note that $\omega_N^{N, N}$ is an exact form and therefore we may drop the term $a=N$.
In conclusion, we have shown that 
\[
\int_{\gamma} r_3(2)(\xi_N^{1,1}(\e))=\dfrac1{6N} \sum_{1 \leq a <N} \e^{-a} \left(F\left(\dfrac{a}{N};-1 \right) - F\left(1-\dfrac{a}{N};-1 \right)\right) \int_{\gamma}\widetilde{\omega}_N^{a, a}
\]
for all $\gamma\in H_1(X_N(\C),\Z)$,
thereby completing the proof. 
\end{proof}

\begin{cor}\label{cor:indep}
For any $N$ not divisible by $3$, the elements $\Xi_N^{-2,1}$, $\Xi_N^{1,-2}$ and $\Xi_N^{1,1}(\e)$ are non-zero and linearly independent in $K_4^{(3)}(X_{N, K})$ with $K=\Q(\e)$.
In particular, when $N$ is odd, the elements $\Xi_N^{-2,1}$, $\Xi_N^{1,-2}$ and $\Xi_N^{1,1}(-1)$ are non-zero and linearly independent in $K_4^{(3)}(X_N)$.
\end{cor}

\begin{proof}
By Proposition \ref{cor:pb} and Theorem \ref{thm:reg_hyperg}, we have
\begin{align*}
r_3(2)(\xi_N^{-2,1})&=-\frac1{6N} \sum_{1 \leq a < N} \e^{-a} \left(  F\left(\frac{a}N; -1\right) - F\left(1-\frac{a}N; -1\right) \right)(\tau^{-1})^*\widetilde{\omega}_N^{a, a}, \\
r_3(2)(\xi_N^{1,-2})&=\frac1{6N} \sum_{1 \leq a < N} \e^{-a} \left(  F\left(\frac{a}N; -1\right) - F\left(1-\frac{a}N; -1\right) \right)(\tau^{-1}\circ \iota)^*\widetilde{\omega}_N^{a, a}.   
\end{align*}
Observe that 
\begin{equation}\label{pull}
(\tau^{-1})^*\omega^{a,a}_N=\e^a \omega_N^{N-2a,a} \qquad \text{ and } \qquad (\tau^{-1}\circ \iota)^*\omega^{a,a}_N=-\e^a \omega_N^{a,N-2a}. 
\end{equation}
By the assumption on $N$, the differential forms $\omega_N^{a, a}$, $\omega_N^{a, N-2a}$ and $\omega_N^{N-2a, a}$ are linearly independent. The non-triviality part of the statement follows from Theorem \ref{thm:nontrivial} and the second assertion follows by taking $\e=-1$.   
\end{proof}

\begin{rmk}
Alternatively, Corollary \ref{cor:indep} immediately follows from Corollary \ref{cor:pNrs} and Theorem \ref{thm:nontrivial}. 
    Note that when $N=3$, the elements $\Xi_N^{-2,1}$, $\Xi_N^{1,-2}$ and $\Xi_N^{1,1}(-1)$ all belong to the (conjecturally) $1$-dimensional group $K_4^{(3)}(X_3^{[1,1]})\simeq K_4^{(3)}(X_3)$. Note that $\omega_3^{2,2}\equiv -\omega_3^{-1,2}$ modulo exact forms by \cite[Lemma 4.19]{Ots11}. We thus find using \eqref{pull} that 
    \begin{align*}
    r_3(2)(\xi_3^{-2,1})&=-\frac1{18} \left(  F\left(\frac{1}3; -1\right) - F\left(1-\frac{1}3; -1\right) \right) \left( \frac{1}{3} B\left( \frac{1}{3}, \frac{1}{3} \right)  \right)^{-1} \omega_3^{1, 1}\\
    &\quad -\frac1{18} \left(  F\left(\frac{2}3; -1\right) - F\left(1-\frac{2}3; -1\right) \right) \left( \frac{1}{3} B\left( \frac{2}{3}, \frac{2}{3} \right)  \right)^{-1} \omega_3^{-1, 2} \\
    &\equiv -\frac1{18} \left(  F\left(\frac{1}3; -1\right) - F\left(1-\frac{1}3; -1\right) \right) \left( \frac{1}{3} B\left( \frac{1}{3}, \frac{1}{3} \right)  \right)^{-1} \omega_3^{1, 1}\\
    &\quad +\frac1{18} \left(  F\left(\frac{2}3; -1\right) - F\left(1-\frac{2}3; -1\right) \right) \left( \frac{1}{3} B\left( \frac{2}{3}, \frac{2}{3} \right)  \right)^{-1} \omega_3^{2, 2} \\
    & = r_3(2)(\xi_3^{1,1}(-1)).
    \end{align*}
    A similar calculation for $\xi_3^{1,-2}$ leads to the equalities
    \[
    r_3(2)(\xi_3^{1,1}(-1))=r_3(2)(\xi_3^{-2,1})=r_3(2)(\xi_3^{1,-2}).
    \]
    The conjectural injectivity of Beilinson's regulator (as part of Beilinson's conjectures, see Section \ref{subsec:beilinson}) would imply that $\Xi_3^{-2,1}
    =\Xi_3^{1,-2}=\Xi_3^{1,1}(-1)$ in $K_4^{(3)}(X_3)$. It would be interesting (but probably difficult) to prove these relations unconditionally. 
\end{rmk}

\section{Numerical verification of Beilinson's conjecture} \label{sec:num}

In this section, we investigate numerically Beilinson's conjecture for the $L$-values $L(X_N, 3)$ and $L(X_N^{[a,b]}, 3)$. We use for this the element $\xi_N^{-2,1}\in H^2(\Gamma(\Q(X_N),3))$ constructed in Section \ref{sec:elements}. This element has the advantage that $\Xi_N^{-2,1}=\varphi_{\mathrm{dJ}}(\xi_N^{-2,1})$ belongs to $K_4^{(3)}(X_N)$ no matter the parity of $N$ by Theorem \ref{prop:res1}.

After recalling in Section \ref{subsec:beilinson} the statement of the Beilinson conjecture, we explain in Section \ref{subsec:regnum} how to compute numerically the regulator integral
\begin{equation*}
    R_N = \int_{\gamma_N} r_3(2)(\xi_N^{-2,1}).
\end{equation*}
We then examine the relation with $L$-values in Section \ref{subsec:conjnum}, and check numerically Beilinson's conjecture in the cases $N=3,4,6$.

We assume in this section that $\z_N=e^{\frac{2\pi i}{N}}$ and $\e=e^{\frac{\pi i}{N}}$.

\subsection{Beilinson's conjecture} \label{subsec:beilinson}

Let $X$ be a smooth, projective, geometrically connected curve of genus $g$ defined over $\Q$. Consider the $L$-function $L(X,s):=L(h^1(X), s)$ associated with the compatible family of $\ell$-adic $\mathrm{Gal}(\overline{\Q}/\Q)$-representations $\{ H^1_{\mathrm{et}}(X_{\overline{\Q}}, \Q_\ell) \}_\ell$. It is given by a convergent Euler product in the region $\Re(s)> 3/2$. Conjecturally, $L(X,s)$ has an analytic continuation to $\C$ and satisfies a functional equation relating $L(X,s)$ and $L(X,2-s)$, \emph{i.e.}~centered at $s=1$ (see \cite[Section 1]{Nek94}).

For all integers $n \geq 2$, Beilinson \cite{Bei80, Nek94} has defined regulator maps 
\begin{equation}\label{rb(n)}
    r_B(n) \colon K_{2n-2}^{(n)}(X) \longrightarrow H^2_{\mathcal{D}}(X_{\mathbb{R}}, \mathbb{R}(n))
\end{equation}
from the $n$-th Adams eigenspace of the $K$-group $K_{2n-2}(X)\otimes \Q$ to the Deligne cohomology. Note that $K_{2n-2}^{(n)}(X)$ is isomorphic to the motivic cohomology group $H^2_{\mathcal{M}}(X, \Q(n))$ \cite[Section 5.1]{Nek94}.
Beilinson \cite{Bei80}, \cite[(6.1)]{Nek94} has conjectured that $r_B(n)_{\Z}\otimes \R$ \footnote{Here $r_B(n)_\Z$ is the restriction of the map $r_B(n)$ to the integral subspace $K_{2n-2}^{(n)}(X)_{\Z}$ of $K_{2n-2}^{(n)}(X)$ \cite[(6.1)]{Nek94}. In fact $K_{2n-2}^{(n)}(X)_{\Z} = K_{2n-2}^{(n)}(X)$ for $n \geq 3$, see \cite[Remarks 5.6]{DS91}.} is an isomorphism 
whose determinant (in suitable bases) equals $(2\pi i)^{-2(n-1)g} L(X,n)$ up to a non-zero $\Q$-rational factor. 

We are interested in this conjecture when $n=3$. The functional equation implies that $L(X,s)$ has a zero of order $g$ at $s=-1$ and that $L^{(g)}(X,-1)$ is a rational multiple of $\pi^{-4g} L(X,3)$. 
We will state the conjecture for $L^{(g)}(X,-1)$.
Recall from Section \ref{sec:asy} the regulator map \eqref{BeiReg}
\begin{equation*}
    r_B : K_4^{(3)}(X) \longrightarrow H^1(X(\C), \R(2))^+.
\end{equation*}
It coincides with the regulator map $r_B(3)$ \eqref{rb(n)} after identifying the relevant Deligne cohomology group with $H^1(X(\C), \R(2))^+$. Let $H_1(X(\C), \Q)^+$ denote the fixed part of the first homology group under the action of complex conjugation on $X(\C)$.

\begin{conj}[Beilinson] \label{conj:Beilinson}
    The following statements hold:
    \begin{enumerate}
        \item $K_4^{(3)}(X)$ is a $\Q$-vector space of dimension $g$.
        \item If $(\Xi_1, \ldots, \Xi_g)$ is a $\Q$-basis of $K_4^{(3)}(X)$ and $(\gamma_1, \ldots, \gamma_g)$ is a $\Q$-basis of $H_1(X(\C), \Q)^+$, then
        \begin{equation} \label{eq:Lvalue-Reg}
            L^{(g)}(X,-1) = c \pi^{-2g} \det \Bigl(\int_{\gamma_i} r_B(\Xi_j) \Bigr)_{1 \leq i,j \leq g}
        \end{equation}
        for some $c \in \Q^\times$.
    \end{enumerate}
\end{conj}

Conjecture \ref{conj:Beilinson} is currently out of reach, as we don't even know how to prove that $K_4^{(3)}(X)$ is finite dimensional as soon as $g \geq 1$. In practice, we aim for a weakened version: namely, there should exist $\Xi_1, \ldots, \Xi_g$ in $K_4^{(3)}(X)$ such that \eqref{eq:Lvalue-Reg} holds. This weak conjecture has been proved by Beilinson for modular curves \cite{Bei86}. Since elliptic curves over $\Q$ are modular, this implies the weak conjecture for all elliptic curves over $\Q$. 

Thanks to de Jeu's Theorem \ref{thm:reg_comp}, we can reformulate Conjecture \ref{conj:Beilinson} in terms of Goncharov's regulator map $r_3(2)$, simply replacing $\Xi_1, \ldots \Xi_g$ by classes $\xi_1, \ldots, \xi_g$ in $H^2(\Gamma(X,3))$ in the equation \eqref{eq:Lvalue-Reg}.

There is also an extension of Beilinson's conjectures to Chow motives, see {\it e.g.}~\cite[(6.7)]{Nek94}. For simplicity, we only state the weak version of Conjecture \ref{conj:Beilinson} for the motives $X_N^{[a,b]}$ defined in Section \ref{subsec:motives} in the case where $(a,b,N)=1$. As we now recall, the $L$-function of $X_N^{[a,b]}$ can be expressed in terms of $L$-functions of Hecke characters. For a prime $v \nmid N$ of $\Q(\zeta_N)$, let $\mathbb{F}_v$ denote the residue field at $v$, and let $\chi_v \colon \mathbb{F}_v^{\times} \to \mu_N$ be the $N$-th power residue character defined by $\chi_v(x) \equiv x^{(|\mathbb{F}_v|-1)/N} \pmod{v}$. 
Consider the {\it Jacobi sum} 
$$j_N^{a, b}(v)=-\sum_{\substack{x, y \in \mathbb{F}_v^{\times}\\x+y=1}} \chi_N^a(x) \chi_N^b(y). $$
After fixing an embedding $\mu_N\hookrightarrow \C^\times$, this defines a Hecke character. 
By \cite[Theorem 3.9]{Ots11}, the $L$-function $L(X_N^{[a, b]}, s)$, which is defined as the $L$-function of the compatible family of $l$-adic Galois representations $(p_N^{[a, b]})^*H^1_{\rm et}(X_{N,\overline{\Q}} , \Q_l)$ (for $l \nmid N$), 
agrees with the Hecke $L$-function of $j_N^{a,b}$: 
\begin{equation}\label{eq:HeckeL}
    L(X_N^{[a, b]}, s)=L(j_N^{a, b}, s).
\end{equation} 
As a corollary of this fact, $L(X_N^{[a, b]}, s)$ can be analytically continued to the whole complex plane and satisfies a functional equation relating $L(X_N^{[a, b]},s)$ and $L(X_N^{[a, b]},2-s)$ by Hecke's theory.

In accordance with Section \ref{subsec:motives}, we define the plus part of the homology of $X_N^{[a, b]}$ as follows:
$$H_1(X_N^{[a, b]}, \Q)^+:= (p_N^{[a, b]})_*H_1(X_N(\C), \Q)^+. $$
This is a $\Q$-vector space of dimension $\varphi(N)/2$ by Corollary \ref{cor:rank}. 

\begin{conj}[Beilinson]\label{conj:weakBei}
    Let $N \geq 3$, $(a,b) \in I_N$, and $k := \varphi(N)/2$. Then there exist $\Xi_1, \ldots, \Xi_{k}$ in $K_4^{(3)}(X_N^{[a,b]})$, a $\Q$-basis $(\gamma_1, \ldots, \gamma_{k})$ of $H_1(X_N^{[a,b]}, \Q)^+$, and a constant $c \in \Q^\times$ such that
    \begin{equation*} 
        L^{(k)}(X_N^{[a,b]},-1) = c \pi^{-2k} \det \Bigl(\int_{\gamma_i} r_B(\Xi_j) \Bigr)_{1 \leq i,j \leq k}.
    \end{equation*}
\end{conj}

\subsection{Numerical computation of $R_N$} \label{subsec:regnum}

Recall that 
\begin{equation*}
    R_N = \int_{\gamma_N} r_3(2)(\xi_N^{-2,1}).
\end{equation*}

It will cost no more effort to deal with the twists of $\xi_N^{-2,1}$ by $N$-th roots of unity. Let $\zeta$ be any $N$-th root of unity. 
Writing $\zeta = \zeta_N^b$, Proposition \ref{prop:trace}(i) implies $\xi_N^{-2,1}(\zeta_N^b) = (g_N^{0,b})^* \xi_N^{-2,1}$, so that $\Xi_N^{-2,1}(\zeta) \in K_4^{(3)}(X_{N,\Q(\zeta_N)})$ by Theorem \ref{prop:res1}.

By Proposition \ref{iota}(iii), we have $\tau^* \xi^{-2, 1}_N(\zeta)=-\xi^{1, 1}_N(\e \zeta)$. Therefore
\begin{equation}\label{eq:num}
    \int_{\gamma_N} r_3(2)(\xi^{-2,1}_N(\zeta)) = - \int_{\gamma_N} r_3(2)((\tau^{-1})^* \xi^{1,1}_N(\e \zeta)) = - \int_{(\tau^{-1})_* \gamma_N} r_3(2)(\xi^{1,1}_N(\e \zeta)).
\end{equation}

We now express $(\tau^{-1})_*(\gamma_N)$ in terms of $\gamma_N$.

\begin{lem} \label{lem:sumroots}
For an integer $0<j<N$, we have
\begin{equation*}
    \frac{1}{N} \sum_{k=1}^{N-1} k \zeta_N^{jk} = \frac{1}{\zeta_N^j-1}.
\end{equation*}
\end{lem}

\begin{proof}
    This follows from differentiating the identity
    \begin{equation*}
        (1-x) \sum_{k=0}^{N-1} x^k = 1-x^N
    \end{equation*}
    and evaluating at $x = \zeta_N^j$.
\end{proof}

\begin{lem} \label{lem:tau gamma}
We have
\begin{equation*}
    (\tau^{-1})_*(\gamma_N) = \sum_{k=1}^{N-1} \frac{k}{N} g^{k+1,k}(\gamma_N) - \frac{k}{N} g^{k,k}(\gamma_N).
\end{equation*}
\end{lem}

\begin{proof}
To identify the cycle $(\tau^{-1})_*(\gamma_N)$, we pair it with the basis of $\Omega^1(X_N(\C))$ given by \eqref{Omega1}. For $a,b>0$, $a+b<N$, we have
\begin{equation*}
    (\tau^{-1})^* \omega_N^{a, b} = \left(\dfrac1x\right)^a \left( \dfrac{\e^{-1} y}x\right)^{b-N} \dfrac{d(1/x)}{1/x} = \e^{-b} x^{N-a-b} y^{b-N} \dfrac{dx}{x}= \e^{-b} \omega_N^{N-a-b, b}.
\end{equation*}
Therefore, for these values of $a,b$,
\begin{equation}\label{eq:tau gamma}
    \langle (\tau^{-1})_*(\gamma_N), \omega_N^{a,b} \rangle = \langle \gamma_N, (\tau^{-1})^* \omega_N^{a,b} \rangle = \langle \gamma_N, \e^{-b} \omega_N^{N-a-b,b} \rangle = \frac{\e^{-b}}{N} B\Bigl(\frac{N-a-b}{N}, \frac{b}{N} \Bigr),
\end{equation}
the last equality following from \eqref{eq:normint}. The expression with the Beta function can be transformed using the reflection formula $\Gamma(x) \Gamma(1-x)= \frac{\pi}{\sin \pi x}$ for the Gamma function:
\begin{equation}\label{eq2:tau gamma}
\begin{split}
    B\Bigl(\frac{N-a-b}{N}, \frac{b}{N} \Bigr) & = \frac{\Gamma\bigl(1-\frac{a}{N}-\frac{b}{N}\bigr) \Gamma\bigl(\frac{b}{N}\bigr)}{\Gamma\bigl(1-\frac{a}{N}\bigr)} \\
    & = \frac{\Gamma\bigl(\frac{a}{N}\bigr) \Gamma\bigl(\frac{b}{N}\bigr)}{\Gamma\bigl(\frac{a+b}{N}\bigr)} \frac{\sin\bigl(\pi \frac{a}{N}\bigr)}{\sin\bigl(\pi \frac{a+b}{N}\bigr)} \\
    & = B\Bigl(\frac{a}{N}, \frac{b}{N}\Bigr) \e^b \frac{\zeta_N^a-1}{\zeta_N^{a+b}-1}.
\end{split}
\end{equation}
By Lemma \ref{lem:sumroots}, we deduce that
\begin{equation}\label{eq3:tau gamma}
    \frac{\zeta_N^a-1}{\zeta_N^{a+b}-1} = \frac{1}{N} \sum_{k=1}^{N-1} k \zeta_N^{a(k+1)+bk} - k \zeta_N^{ak+bk}.
\end{equation}
Combining \eqref{eq:tau gamma}, \eqref{eq2:tau gamma} and \eqref{eq3:tau gamma}, we obtain
\begin{equation*}
    \langle (\tau^{-1})_*(\gamma_N), \omega_N^{a,b} \rangle = \frac{1}{N} \Bigl(\sum_{k=1}^{N-1} k \zeta_N^{a(k+1)+bk} - k \zeta_N^{ak+bk}\Bigr) \langle \gamma_N, \omega_N^{a,b} \rangle.
\end{equation*}
The result then follows from the relation $(g_N^{r,s})^* \omega_N^{a,b} = \zeta_N^{ar+bs} \omega_N^{a,b}$.
\end{proof}

Our final formula to evaluate the regulator of $\xi_N^{-2,1}(\zeta)$ reads as follows.

\begin{prop}\label{prop:regnum}
Let $\zeta$ be any $N$-th root of unity. Then
\begin{equation*}
    \int_{\gamma_N} r_3(2)(\xi^{-2,1}_N(\zeta)) = \sum_{k=1}^{N-1} \frac{k}{N} \int_{\delta_N} \Bigl(r_3(2)(\xi_N^{1,1}(\e \zeta \zeta_N^{2k})) - r_3(2)(\xi_N^{1,1}(\e \zeta \zeta_N^{2k+1})) \Bigr).
\end{equation*}
\end{prop}

\begin{proof}
Using Equation \eqref{eq:num} and Lemma \ref{lem:tau gamma}, our regulator is expressed as
\begin{align*}
    \int_{\gamma_N} r_3(2)(\xi^{-2,1}_N(\zeta)) & = - \left(\sum_{k=1}^{N-1} \frac{k}{N} \int_{\gamma_N} r_3(2)((g^{k+1,k})^* \xi_N^{1,1}(\e \zeta)) - \frac{k}{N} \int_{\gamma_N} r_3(2)((g^{k,k})^* \xi_N^{1,1}(\e \zeta)) \right) \\
    & = \sum_{k=1}^{N-1} \frac{k}{N} \int_{\gamma_N} \Bigl(r_3(2)(\xi_N^{1,1}(\e \zeta \zeta_N^{2k})) - r_3(2)(\xi_N^{1,1}(\e \zeta \zeta_N^{2k+1}) \Bigr),
\end{align*}
where for the last equality we used Proposition \ref{prop:trace}(i). The integral over $\gamma_N$ is expressed as an integral over $\delta_N$ by Lemma \ref{lem:reg_xi1}. Noting that
\begin{equation*}
\sum_{r \in \Z/N\Z} \xi_N^{1,1}(\e \zeta \zeta_N^{2k} \zeta_N^r) = \sum_{r \in \Z/N\Z} \xi_N^{1,1}(\e \zeta \zeta_N^{2k+1} \zeta_N^r),
\end{equation*}
we obtain the desired result.
\end{proof}

To compute numerically the right-hand side in Proposition \ref{prop:regnum}, we simply use Lemma \ref{lem:reg_xi2} and PARI/GP's function $\texttt{intnum}$. This gives us the following values of $R_N$ for $3 \leqslant N \leqslant 6$.
\begin{equation}\label{table}
\begin{tabular}{ c|c|c|c|c } 
 $N$ & 3 & 4 & 5 & 6 \\ 
 \hline
 $R_N$ & 7.14926\dots & 10.19559\dots & 17.68165\dots & 23.48783\dots \\ 
\end{tabular}
\end{equation}
The fact that these values are non-zero confirm that $\Xi_N^{-2,1}\neq 0$ as established in Theorem \ref{thm:nontrivial}.

For any $a,b \in \Z/N\Z$, we further explain how to compute
\begin{equation*}
    R_N^{a,b} = \int_{\gamma_N} r_3(2)(p_N^{a,b} \xi_N^{-2,1}).
\end{equation*}

\begin{prop}\label{prop:regR}
Let $a,b \in \Z/N\Z$. The following statements hold:
\begin{enumerate}
    \item If $a+2b \neq 0$, then $R_N^{a,b} = 0$.
    \item If $2b=0$, then $R_N^{0,b}=0$.
    \item If $2b \neq 0$, then
    \begin{equation*}
        R_N^{-2b, b} = - \frac{1}{N(\zeta_N^b + 1)} \sum_{u \in \Z/N\Z} \zeta_N^{-bu} \int_{\delta_N} r_3(2)(\xi_N^{1,1}(\e \zeta_N^u)).
    \end{equation*}
\end{enumerate}
\end{prop}

\begin{proof}
    If $a+2b \neq 0$ then the projector $p_N^{a,b}$ is orthogonal to every projector $p_N^{a',b'}$ with $a'+2b'=0$. Proposition \ref{projection} then implies that $p_N^{a,b} \xi_N^{-2,1}=0$, which proves (i).
    
    Assume now $a+2b=0$. Using the definition of $p_N^{a,b}$ (Definition \ref{def:pNab}) and Proposition \ref{prop:trace}(i), we compute
    \begin{equation*}
        R_N^{a,b} = \frac{1}{N^2} \sum_{r,s \in \Z/N\Z} \zeta_N^{-ar-bs} \int_{\gamma_N} r_3(2)(\xi_N^{-2,1}(\zeta_N^{-2r+s})).
    \end{equation*}
    Making the change of indices $s \mapsto t = s-2r$, we deduce
    \begin{align*}
        R_N^{a,b} & = \frac{1}{N^2} \Bigl(\sum_{r \in \Z/N\Z} \zeta_N^{-(a+2b)r} \Bigr) \sum_{t \in \Z/N\Z} \zeta_N^{-bt} \int_{\gamma_N} r_3(2)(\xi_N^{-2,1}(\zeta_N^t)) \\
        & = \frac{1}{N} \sum_{t \in \Z/N\Z} \zeta_N^{-bt} \int_{\gamma_N} r_3(2)(\xi_N^{-2,1}(\zeta_N^t)).
    \end{align*}
    Applying Proposition \ref{prop:regnum}, we get
    \begin{align*}
        R_N^{a,b} & = \frac{1}{N} \sum_{t \in \Z/N\Z} \zeta_N^{-bt} \sum_{k=1}^{N-1} \frac{k}{N} \int_{\delta_N} r_3(2)(\xi_N^{1,1}(\e \zeta_N^{t+2k})) - r_3(2)(\xi_N^{1,1}(\e \zeta_N^{t+2k+1})).
    \end{align*}
    Splitting the integral and making the change of indices $u = t+2k$, respectively $u = t+2k+1$, we obtain
    \begin{align}
    \nonumber
    R_N^{a,b} & = \frac{1}{N} \sum_{k=1}^{N-1} \frac{k}{N} \sum_{u \in \Z/N\Z} \bigl(\zeta_N^{-b(u-2k)} - \zeta_N^{-b(u-2k-1)}\bigr) \int_{\delta_N} r_3(2)(\xi_N^{1,1}(\e \zeta_N^u)) \\
    \label{eq:regR}
    & = \frac{1-\zeta_N^b}{N} \Bigl(\sum_{k=1}^{N-1} \frac{k}{N} \zeta_N^{2bk} \Bigr) \sum_{u \in \Z/N\Z} \zeta_N^{-bu} \int_{\delta_N} r_3(2)(\xi_N^{1,1}(\e \zeta_N^u)).
    \end{align}
    In case (iii) we conclude by applying Lemma \ref{lem:sumroots}.
    
    It remains to prove (ii). If $b=0$, the expression \eqref{eq:regR} shows that $R_N^{0,0}=0$. Assume $b \neq 0$, so that $N$ is even and $b=N/2$. Write $N=2M$. Note that $\z_{M}=e^{\frac{2\pi i}{M}}=\z_N^2$.
    Splitting the sum \eqref{eq:regR} according to the parity of $u$, it suffices to show that 
    \begin{equation}\label{eq:sumvM}
    \sum_{v\in \Z/M\Z} \int_{\delta_N} r_3(2)(\xi^{1,1}_N(\e \zeta_N^{2v}))=\sum_{v\in \Z/M\Z} \int_{\delta_N} r_3(2)(\xi^{1,1}_N(\e \zeta_N^{2v+1})).
    \end{equation}
    By Proposition \ref{dist:cocycle}, we have 
    \[
    \sum_{v\in \Z/M\Z} \xi_N^{1,1}(\e\z_N^{2v})=\sum_{v\in \Z/M\Z} \xi_N^{1,1}(\e\z_M^v)=\pi_{N,2}^*(\xi_{2}^{1,1}(\e^M))
    \]
    on the one hand and
    \[
    \sum_{v\in \Z/M\Z} \xi_N^{1,1}(\e\z_N^{2v+1})=\sum_{v\in \Z/M\Z} \xi_N^{1,1}(\e\z_M^v \z_N)=\pi_{N,2}^*(\xi_{2}^{1,1}(\e^M \z_N^M))
    \]
    on the other hand.
    Observe that $\e^M=i$ and $\z_N^M=-1$.
    We are thus trying to show that 
    \[
    \int_{\delta_N} r_3(2)(\pi_{N,2}^*(\xi_{2}^{1,1}(i)))=\int_{\delta_N} r_3(2)(\pi_{N,2}^*(\xi_{2}^{1,1}(-i))).
    \]
    This equality holds since the regulator map $r_3(2)$ is defined as a real part of some differential form. 
    
    For the benefit of the reader, we include the details. According to Lemma \ref{lem:reg_xi2}, given any root of unity $\z$, we have
    \begin{align*}
        \int_{\delta_N} r_3(2)(\xi_N^{1,1}(\zeta)) = \int_0^1 \rho_{N,\z}^1 + \rho_{N,\z}^2,
    \end{align*}
    where
    \begin{align*}
        \rho_{N,\z}^1 & = - \frac13 \log \bigl|1-\zeta (t(1-t))^{1/N}\bigr| \Bigl(-\frac{\log(1-t)}{t} - \frac{\log t}{1-t}\Bigr) dt, \\
        \rho_{N,\z}^2 & = \frac13 \log t \Bigl(-\log \bigl|1-\zeta (t(1-t))^{1/N}\bigr| \dlog(t(1-t)) + \log(t(1-t)) \dlog \bigl|1-\zeta (t(1-t))^{1/N}\bigr| \Bigr).
    \end{align*}
    Note that 
    \begin{multline*}
\sum_{v\in \Z/M\Z} \log \bigl|1-\e \zeta_N^{2v} (t(1-t))^{1/N}\bigr| = \log \left\vert \prod_{v\in \Z/M\Z} \Big( 1- \zeta_M^{v} \e(t(1-t))^{1/(2M)}\Big)  \right\vert \\  = \log \left\vert 1- \e^M \sqrt{t(1-t)}  \right\vert  = \log \left\vert 1- i \sqrt{t(1-t)}  \right\vert
\end{multline*}
and 
\begin{multline*}
\sum_{v\in \Z/M\Z} \log \bigl|1-\e \zeta_N^{2v+1} (t(1-t))^{1/N}\bigr| = \log \left\vert \prod_{v\in \Z/M\Z} \Big( 1- \zeta_M^{v} \e \z_N (t(1-t))^{1/(2M)}\Big)  \right\vert \\  = \log \left\vert 1- \e^M \z_N^M \sqrt{t(1-t)}  \right\vert  = \log \left\vert 1 + i \sqrt{t(1-t)}  \right\vert.
\end{multline*}
The desired equality \eqref{eq:sumvM} then follows from the equality
\[
\log \left\vert 1- i \sqrt{t(1-t)}  \right\vert = \log \left\vert 1 + i \sqrt{t(1-t)}  \right\vert. \qedhere
\]
\end{proof}

\subsection{Checking the conjecture} \label{subsec:conjnum}

The aim of this section is to prove Theorem \ref{thm7}.
The element $\Xi_N^{-2,1}$ lives in a specific part of the motivic cohomology of $X_N$. Indeed, by Proposition \ref{projection} and functoriality of $\varphi_{\mathrm{dJ}}$, we have in $K_4^{(3)}(X_N) \otimes \C$
\begin{equation*}
    \Xi_N^{-2,1} = \sum_{b=0}^{N-1} p_N^{-2b,b} \Xi_N^{-2,1}.
\end{equation*}
The most interesting part of the decomposition above comes from the terms $b \in (\Z/N\Z)^\times$, contributing to the element $p_N^{[-2,1]}(\Xi_N^{-2,1})$ of $K_4^{(3)}(X_N^{[-2,1]})$. The other terms come from a Fermat curve of lower degree. Namely, let $d$ be a positive divisor of $N$. Then
\begin{equation*}
    \sum_{\substack{b = 0 \\ (b,N)=d}}^{N-1} p_N^{-2b,b} = p_N^{[-2d, d]}
\end{equation*}
and the motive $X_N^{[-2d,d]}$ is isomorphic to $X_{N/d}^{[-2,1]}$ (see \cite[(2.3)]{Ots15}). This motive is trivial for $d=N$, as well as for $d=N/2$ in the case $N$ is even, since the Fermat curves $X_1$ and $X_2$ have genus $0$.

    For the numerical verification, we focus on those motives $X_N^{[-2,1]}$ that have rank $2$ (see \eqref{list:rank2} where all rank $2$ motives of the form $X_N^{[a,b]}$ are listed). 
    Then there exist elliptic curves $E_{N}$ over $\Q$ with CM by $\Q(\z_N)$ such that $X_N^{[-2,1]}\simeq h^1(E_N)$ as motives with $\Q$-coefficients. 
    In order to compute the $L$-functions in the relevant cases for us, we will give explicit decriptions of $E_3$, $E_4$, and $E_6$ below (up to isogeny, which is enough to compute the $L$-functions). Alternatively, we could use \eqref{eq:HeckeL} to compute them as Hecke $L$-functions (using existing implementations in MAGMA \cite{Magma} for example). The latter approach is better suited for higher rank cases.

\begin{prop} \label{elliptic_curve_isom}    
    There exist isogenies defined over $\Q$: 
    \begin{align*}
        & E_3 \sim 27a1  : v_0^2w_0=u_0^3-432 w_0^3,\\
        & E_4 \sim 32a2  : v_0^2w_0=u_0^3 - u_0w_0^2,\\
        & E_6 \sim 108a1  : v_0^2w_0= u_0^3+4w_0^3, \\
    \end{align*}
    where we use the labels from Cremona's tables of elliptic curves \cite{Cre97}.
\end{prop}

\begin{proof}
The first two cases are well-known (see \cite[Sections 5.2 and 5.3]{Ots11_2}), so we will prove the last case. Let $E$ be the elliptic curve given by the affine equation $v^2=u^3+4$.
Consider the morphism
\begin{align*}
& p\colon X_6 \longrightarrow E, \qquad  (x,y) \longmapsto \left(\frac{x^4}{y^2}, y^3+\frac{1}{y^3}\right).
\end{align*}
Observe that
$$p^*\left(\dfrac{du}{2v}\right) = - y^1 x^{-2} \frac{dy}{y} =  \omega_6^{4,1} \qquad \text{and} \qquad p^*\left(\dfrac{u \, du}{2v}\right) = -y^{-1} x^2 \frac{dy}{y} = \omega_6^{8, -1} \equiv -2\omega_6^{2, 5},$$
where we used \cite[Lemma 4.19]{Ots11}. By \eqref{basiseq}, we see that $p^*(H^1_{\rm dR}(E))= H^1_{\rm dR}(X_6^{[-2, 1]})$. Since $p$ is defined over $\Q$, this implies that it induces an isomorphism $X_6^{[-2,1]}\simeq h^1(E)$.
\end{proof}

\subsubsection{Case $N=3$}

In this case $X_3$ is isomorphic to the elliptic curve 27a1 and isogenous to the elliptic curve $E_{3}$ by Proposition \ref{elliptic_curve_isom}. Using the functions \texttt{lfun} and \texttt{lindep} of PARI/GP \cite{PARI}, we find the relation
\begin{equation*}
    R_3 \stackrel{?}{=} -\frac45 \pi^2 L'(E_{3},-1).
\end{equation*}
Here $\stackrel{?}{=}$ means an equality up to at least 35 digits. Taking into account the factor $\frac12$ in de Jeu's Theorem \ref{thm:reg_comp}, this confirms numerically Beilinson's conjecture for $L'(E_{3},-1)$ and proves the first case of Theorem \ref{thm7}.

\begin{rmk}\label{rmk:jeuX3}
    In \cite[p.~19]{JeuSlides}, de Jeu considers the elliptic curve $E \colon u(1-u)=(-v)^3$ (Cremona label 27a3), which is isomorphic to $C^{1,1}_3$ over $\Q$ and $3$-isogenous to the elliptic curve $X_3$ (Cremona label 27a1) over $\Q$. An explicit $3$-isogeny $\psi \colon X_3 \to E$ is given by $(x,y)\mapsto (u,v)=(x^3, -xy)$. De Jeu produces an element in $K_4^{(3)}(E)$ by constructing an explicit degree $2$ cocycle in $\widetilde{\mathcal{M}}_{(3)}^{\bullet}(\Q(E))$. Translated into the Goncharov complex, the cocycle constructed by de Jeu is:
\[
\beta:=\{ u \}_2 \otimes (1-v)-3\{ v \}_2 \otimes u\in B_2(\Q(E))\otimes \Q(E)^\times_{\Q}.
\]
Pulling it back to $X_3$ via $\psi$ recovers our cocycle
\[
\xi_3^{1,1}(-1)=\{ x^3 \}_2 \otimes (1+xy)-3\{ -xy \}_2 \otimes x^3 \in B_2(\Q(X_3))\otimes \Q(X_3)^\times_{\Q},
\]
and this is the observation that originally inspired our construction in Section \ref{s:2coc}.

De Jeu integrates $r_B(\varphi_{\rm dJ}(\beta))$ over a $\Z$-basis of $H_1({E(\C), \Z})^+$. 
By \cite[Corollary 2.3]{Ots}, $\gamma_3$ is not integral but we have
$3 \gamma_3 \in H_1(X_3(\C), \Z)$. This element gives the desired $\Z$-basis of $H_1({E(\C), \Z})^+$ after applying $\psi_*$. In fact, it can be shown that $$\{(\pi_3^{1,1})_* (3\gamma_3), \zeta_3\cdot (\pi_3^{1,1})_*(3\gamma_3)\}$$ is a symplectic basis of $H_1(C_3^{1,1}(\C), \Z)$ as we now explain. Here, $\zeta_3\cdot (\pi_3^{1,1})_*(3\gamma_3)$ is the cycle induced by the action $\zeta_3 \cdot (u, v):=(u, \zeta_3 v)$.  
For $\gamma$, $\gamma' \in H_1(X_3(\C), \Z)$, let $\gamma \sharp \gamma' \in \Z$ denote the intersection number.  
Note that $X_3$ is generically Galois over $C_3^{1,1}$ and $\operatorname{Gal}(X_3/C_3^{1,1}) =\langle g_3^{1, 2}\rangle$. 
Therefore we have 
\begin{align} \label{int1}
(\pi_3^{1,1})_*(3\gamma_3) \sharp (\zeta_3 \cdot (\pi_3^{1,1})_*(3\gamma_3) )&= 9\sum_{r \in \Z/3\Z}(g_3^{r, 2r} \gamma_3) \sharp (g_3^{0, 1} \gamma_3) 
\end{align}
by the projection formula. 
It follows from \cite[Corollary 3.4]{Ots} that the right hand side of \eqref{int1} is equal to $1$, 
which means that 
$\{(\pi_3^{1,1})_* (3\gamma_3), \zeta_3\cdot (\pi_3^{1,1})_*(3\gamma_3)\}$ is a symplectic basis of $H_1(C_3^{1,1}(\C), \Z)$. 
Therefore, de Jeu's computation shows that
$$(2\pi i)^{-2} \int_{\psi_*(3 \gamma_3)} r_B(\varphi_{\rm dJ}(\beta))\overset{?}{=} \frac{6}5 L'(E, -1), $$
hence 
$$\int_{\gamma_3} r_B(\varphi_{\rm dJ}(\xi_3^{1,1}(-1)))=\int_{\psi_*(\gamma_3)} r_B(\varphi_{\rm dJ}(\beta))\overset{?}{=}-\frac{8}5 \pi^2 L'(E, -1). $$
This matches our numerical calculation of $R_3$ by Theorem \ref{thm:reg_comp}, which is valid up to sign.   
\end{rmk}

\subsubsection{Case $N=4$}

By the above discussion $\xi_4^{-2,1} = p_4^{[-2,1]} \xi_4^{-2,1}$. The motive $X_4^{[-2,1]}$ has rank $2$ by Corollary \ref{cor:rank} and $L(X_4^{[-2,1]},s)=L(E_4, s)$ by Proposition \ref{elliptic_curve_isom}. We therefore expect a relation between $R_4$ and $L'(E_{4}, -1)$, and indeed we find, using PARI/GP's functions \texttt{lfun} and \texttt{lindep}, that
\begin{equation*}
    R_4 \stackrel{?}{=} -\frac45 \pi^2 L'(E_{4},-1),
\end{equation*}
thereby proving the second case of Theorem \ref{thm7}.

\subsubsection{Case $N=6$}

In this case $\xi_6^{-2,1} = \bigl(p_6^{[-2,1]} + p_6^{[-4,2]}\bigr)(\xi_6^{-2,1})$. The motive $X_6^{[-2,1]}$ has rank $2$ by Corollary \ref{cor:rank} and $L(X_6^{[-2,1]},s)=L(E_6, s)$ by Proposition \ref{elliptic_curve_isom}. The old part $X_6^{[-4,2]} \simeq X_3^{[-2,1]}$ is isomorphic to $h^1(E_{3})$ as in the case $N=3$. In accordance with this decomposition of $\xi_6^{-2,1}$, we find that
\begin{equation*}
    R_6 \stackrel{?}{=} -\frac2{15} \pi^2 L'(E_{6},-1) - \frac25 \pi^2 L'(E_{3},-1).
\end{equation*}

We can be more precise, and check numerically Conjecture \ref{conj:weakBei} for $E_{6}$. 
To this end, we use Proposition \ref{prop:regR}(iii) in the case $N = 6$ and $(a,b)= \pm (-2,1)$ to find that
\begin{equation*}
    \int_{\gamma_6} r_3(2)(p_6^{[-2,1]} \xi_6^{-2,1}) = R_6^{-2,1} + R_6^{2,-1} = 19.9132\dots \stackrel{?}{=} - \frac2{15} \pi^2 L'(E_{6},-1),
\end{equation*}
which completes the proof of Theorem \ref{thm7}. 

\begin{rmk}We end the paper with a few of remarks.
\begin{itemize}  
    \item In the case $N=5$, the rank of the motive $X_5^{[-2,1]}$ is $4$ by Corollary \ref{cor:rank}. Therefore checking Beilinson's conjecture requires two linearly independent elements of $K_4^{(3)}(X_5^{[-2,1]})$, and we only constructed one element. By Remark \ref{rmk:Cmot} (valid for $N$ prime), we have $X_5^{[-2,1]}\simeq h^1(C_5^{1,3})$. Here, $C_5^{1,3}$ is the cyclic Fermat quotient $v^5=u(1-u)^3$ of genus $2$. In particular, it is hyperelliptic. Since its Jacobian $J^{1,3}_5$ is absolutely simple (see \cite{koblitzrohrlich} for an easy criterion for simplicity), we cannot hope to construct another element using endomorphisms. 
    \item In \cite[Theorem 7.2]{Bru20}, the first-named author constructed elements $\xi_1(a, b)\in K_4^{(3)}(X_1(N))$ for $(a,b) \in G_N$. These can be used to define elements of $K_4^{(3)}(E)$ for all elliptic curves $E$ over $\Q$ via modular parametrizations. An interesting but difficult question would be to relate these elements to ours for the elliptic curves $E_3$, $E_4$ and $E_6$. 
\item When the motives $X_N^{[a,b]}$ have rank $2$ (listed in \eqref{list:rank2}), Conjecture \ref{conj:weakBei} is known due to work of Deninger \cite{Deninger90}. This situation is similar to the one encountered by Otsubo in the setting of $K_2$ of Fermat curves. Indeed, he makes a short remark \cite[Remark 4.18]{Ots11} stating that the surjectivity of the regulator map $r_B(2)$ for these rank $2$ motives was already known. Our computations in Section \ref{subsec:conjnum} numerically recover Deninger's results for the motives $X^{[1,1]}_3$, $X_4^{[2,1]}$, and $X^{[4,1]}_6$ using a different construction of $K_4$ elements.
It would be interesting (but probably difficult) to compare Deninger's construction with ours. 
\end{itemize}
\end{rmk}

\end{document}